\documentclass[10pt]{article}
\usepackage{mathrsfs}
\input epsf
\usepackage{amssymb}
\usepackage{amsmath}
\usepackage{amsfonts}
\input amssym.def
\newtheorem{theorem}{Theorem}[section]

\newtheorem{prop}[theorem]{Proposition}

\newtheorem{definition}{Definition}
\newtheorem{re}{Remark}

\newenvironment{proof}{\noindent {\em Proof  }}{\hspace*{1cm}
        \hspace*{\fill}$\rule{1.2ex}{1.4ex}$\vspace{.15cm}}
\begin{document}
\title{\bf On the low regularity of the fifth order
Kadomtsev-Petviashvili I equation}
\author{ Wengu Chen$^1$, Junfeng Li$^{2}$ and Changxing Miao$^1$\\$^1$Institute of Applied Physics
and Computational Mathematics\\P.O.Box 8009, Beijing 100088,
China\\$^2$ School of Mathematical Sciences\\Laboratory of Math and
Complex Systems, Ministry of Education\\Beijing Normal University,
Beijing 100875, P. R. China \hspace{0.5cm}}

\date{E-mail:\,chenwg@iapcm.ac.cn,\,  junfli@yahoo.com.cn,\, miao$_-$changxing@iapcm.ac.cn.}

\maketitle

{\bf Abstract.}\, We consider the fifth order Kadomtsev-Petviashvili
I (KP-I) equation as
$\partial_tu+\alpha\partial_x^3u+\partial^5_xu+\partial_x^{-1}\partial_y^2u+uu_x=0,$
while $\alpha\in \mathbb{R}$. We introduce an interpolated energy
space $E_s$ to consider the well-posedeness of the initial value
problem (IVP) of the fifth order KP-I equation. We obtain the local
well-posedness of IVP of the fifth order KP-I equation in $E_s$ for
$0<s\leq1$. To obtain the local well-posedness, we present a
bilinear estimate in the Bourgain space in the framework of the
interpolated energy space. It crucially depends on the dyadic
decomposed Strichartz estimate,
the fifth order dispersive smoothing effect and maximal estimate. 

\section{Introduction}

We consider the initial value problem (IVP) of the fifth order
Kadomtsev-Petviashvili (KP) equation
\begin{eqnarray}\label{5KP}
\left\{\begin{array}{ll}
\partial_tu+\alpha\partial_x^3u+\beta\partial_x^5u+\partial_x^{-1}\partial_y^2u+u\partial_xu=0, \\
u(0,x,y)=u_0(x,y),\quad\quad(x,y)\in\mathbb{R}^2.\end{array}\right.\end{eqnarray}
Here $\alpha,\beta\in\mathbb{R}$ and $u_0$ is a real valued
function. If $\beta>0$ the equation (\ref{5KP}) is called the fifth
order KP-I and if $\beta<0$ it takes the name the fifth order KP-II.
This equation occurs naturally in the modeling of a long dispersive
wave. Kawahara \cite{Kaw} introduced the fifth order Korteweg-de
Vries equation
\begin{equation}\label{5KdV}
\partial_tu+\alpha\partial_x^3u+\beta\partial_x^5u+u\partial_xu=0,
\end{equation}
which models the wave propagation in one direction. While the KP
equation models the propagation along the x-axis of a nonlinear
dispersive long wave on the surface of a fluid with a slow variation
along the y-axis (see \cite{KadPet,SauTzv1,SauTzv2} and the
references therein).

We begin with a few facts about KP equations. The Fourier transform
of a Schwarz function $f(x,y)$ is defined by
$$\hat{f}(\xi,\mu)=\frac{1}{2\pi}\int_{\mathbb R^2}f(x,\,y)e^{-i(x\xi+y\mu)}dxdy.$$
The dispersive function of the KP equation is
\begin{equation}\label{KPsign}
\omega(\xi,\,\mu)=\beta\xi^5-\alpha\xi^3+\frac{\mu^2}{\xi}.
\end{equation}
The analysis of the IVP of the KP equation depends crucially on the
sign of $\alpha$ and $\beta$.  We take a glance on the case
$\beta=0$. In this case, equation (\ref{5KP}) turns out to be the
third order KP equation. Without loss of generality, we assume
$|\alpha|=1$. If $\alpha=-1$, the equation is called the third order
KP-I equation. While if $\alpha=1$, the equation is called the third
order KP-II equation. By computing the gradient of $\omega$, we get
that for the third order KP-I
\begin{equation}\label{3KP-1sign}
|\nabla\omega(\xi,\mu)|=\Big|\Big(3\xi^2-\frac{\mu^2}{\xi^2},\,2\frac{\mu}{\xi}\Big)\Big|\gtrsim|\xi|.
\end{equation}
For the third order KP-II equation, we have
\begin{equation}\label{3KP-2sign}
|\nabla\omega(\xi,\mu)|=\Big|\Big(-3\xi^2-\frac{\mu^2}{\xi^2},\,2\frac{\mu}{\xi}\Big)\Big|\gtrsim|\xi|^2.
\end{equation}
One can easily recover a full derivative smoothness along the $x$
direction by (\ref{3KP-2sign}), but only a half derivative
smoothness by (\ref{3KP-1sign}). Since the nonlinear term in the
third order KP equation involves a full derivative along the $x$
direction, this explains partially to get the well-posedness for the
IVP of KP-I is much more difficult than that of KP-II.

Another important concept in the analysis of dispersive equation is
the resonance function. Still considering the third order KP
equation, the resonance function is defined by
\begin{equation*}
\begin{split}
R(\xi_1,\xi_2,\mu_1,\mu_2)&=\omega(\xi_1+\xi_2,\mu_1+\mu_2)-\omega(\xi_1,\mu_1)-\omega(\xi_2,\mu_2)\\
&=-\frac{\xi_1\xi_2}{(\xi_1+\xi_2)}\left(3\alpha(\xi_1+\xi_2)^2+
\Big(\frac{\mu_1}{\xi_1}-\frac{\mu_2}{\xi_2}\Big)^2\right).
\end{split}
\end{equation*}
Thus for the third order KP-II equation, we always have the
following inequality
\begin{equation}\label{3KP-2smooth}
|R(\xi_1,\xi_2,\mu_1,\mu_2)|\geq C|\xi_1||\xi_2||\xi_1+\xi_2|.
\end{equation}
However, for the third order KP-I equation, the inequality
(\ref{3KP-2smooth}) is not true all the time. In this case, resonant
interaction happens frequently. The resonant interaction means the
resonance function is zero or close to zero. Generally, we use
(\ref{3KP-2smooth}) to recover the derivative on $x$ by the
regularity on $t$. Thus, the simpler the corresponding zero set, the
easier it is to deal with the problem. This facts also implies that
the well-posedness problem of KP-II is easier than that of KP-I.

A natural function space to consider the well-posedness of the IVP
of the KP equation is the non-isotropic Sobolev space:
\begin{equation}
H^{s_1,s_2}(\mathbb{R}^2):=\Big\{f\in
\mathcal{S}^\prime(\mathbb{R}^2);\big\|<\xi>^{s_1}<\mu>^{s_2}\hat{f}\big\|_{L^2_{\xi,\mu}}<\infty\Big\},
\end{equation}
where $<\xi>=(1+|\xi|).$ Keep in mind that we are still in the case
of $\beta=0$. A scaling argument (e.g. see \cite{SauTzv1}) shows
that $s_1+2s_2>-\frac{1}{2}$ is expected for the local
well-posedness of the IVP of the KP equations in $H^{s_1,s_2}$. As
we pointed out, the third order KP-II has better dispersive effect
than the third order KP-I. The results about the third order KP-II
are very close to the expected indices. In \cite{Bou}, Bourgain
showed the global well-posedness of the third order KP-II in $L^2$,
i.e. $s_1=s_2=0$. This result had been improved by Takaoka and
Tzvetkov \cite{TakTzv} and Isaza and Mej\'ia \cite{IsaMej} to
$s_1>-\frac{1}{3},\, s_2\geq 0$. In \cite{Tak}, Takaoka obtained the
local well-posedness of the IVP of the third order KP-II for
$s_1>-\frac{1}{2},\,s_2=0$ and an additional low frequency condition
$|D_x|^{-\frac{1}{2}+\varepsilon}u_0\in L^2$. Recently, Hadac
\cite{Had} removed the additional condition on the initial value
above. This means in the case $s_2=0$, the result on the third order
KP-II equation is sharp. While for the third order KP-I equation,
the situation is far from the expected. By compactness method,
I\'orio and Nunes \cite{IonNun} obtained the local well-posednes of
the IVP of the third KP-I equation for data in the normal Sobolev
space $H^s(\mathbb{R}^2),\,s>2$ and satisfying a ``zero-mass"
condition. They used only the divergence form of the nonlinearity
and the skew-adjointness of the (linear) dispersion operator. The
condition on $s$ is needed to control the gradient of the solution
in the $L^\infty.$

Another natural space to consider the well-posedness of the IVP of
the KP-I equation is the Energy space. We first notice that the KP
equation (\ref{5KP}) satisfies the following two conversations.
\begin{description}
  \item[Mass] \begin{equation}\label{Mass}\|u\|_{L^2}=\|u_0\|_{L^2}.\end{equation}
  \item[Hamiltonian]\begin{equation}
\label{Hamilton}\begin{split}H(u)=&\frac{\beta}{2}\int(\partial_x^2
u)^2dxdy-\frac{\alpha}{2}\int(\partial_x u)^2dxdy\\
&+\frac{1}{2} \int(\partial_{x}^{-1}\partial_y
u)^2dxdy+\frac{1}{6}\int u^3 dxdy=H(u_0).
\end{split}
\end{equation}
\end{description}
Combining the above two conversations together, we can define the
Energy space for the fifth order KP-I equation ($\beta= 1$) by
\begin{equation}\label{5energy}
E(5th)=\Big\{u\in\mathcal{S}^\prime(\mathbb{R}^2);\|u\|_{E(5th)}=\big\|(1+|\xi|^2+|\xi|^{-1}|\mu|)
\hat{u}(\xi,\eta)\big\|_{L^2}<\infty\Big\}.
\end{equation}
For the third order KP-I equation ($\beta=0,\alpha=-1$), the Energy
space can be defined by
\begin{equation}\label{3energy}
E(3th)=\Big\{u\in\mathcal{S}^\prime(\mathbb{R}^2);\|u\|_{E(3th)}=\big\|(1+|\xi|+|\xi|^{-1}|\mu|)
\hat{u}(\xi,\eta)\big\|_{L^2}<\infty\Big\}.
\end{equation}
On these function spaces, we can prove that for $\beta=1,$
\begin{equation}\label{eq:1.1}\|u(t)\|_{E(5th)}\leq
C\|u_0\|_{E(5th)},\end{equation} and for $\beta=0,\alpha=-1$
\begin{equation}\label{eq:1.2}\|u(t)\|_{E(3th)}\leq
C\|u_0\|_{E(3th)},\end{equation} for any sufficiently smooth
solution $u$ of KP-I equation, uniformly in time (see also
\cite{Col,SauTzv2}). Thus it would be expected to obtain local
well-posedness in this kind of spaces. But the recent results of
Molinet, Saut and Tzvetkov \cite{MolSauTzv1,MolSauTzv2} showed that,
for the third order KP-I $(\beta=0, \,\alpha<0)$, one cannot prove
local well-posedness in any type of non-isotropic $L^2-$based
Sobolev space $H^{s_1,s_2}$, or in the energy space (see also
\cite{KocTzv}), by applying Picard iteration to the integral
equation formulation of the third order KP-I equation. To avoid the
difficulty, one must abandon Picard iteration or find out an
alternative space with similar regularity with $H^{s_1,s_2}$ or
energy space. Recently, Colliander, Ionescu, Kenig and Staffilani
\cite{ColKenSta3} set up the local well-posedness of the IVP of the
third order KP-I equation with small data in the intersection of
energy space $E$ and weighted space $P$ defined by
\begin{equation}\label{3KP-1data}
E=\{f:f\in L^2,\partial_x f\in L^2,\partial_x^{-1}\partial_y f\in
L^2\},\text{ and }P=\{f: (y+i)f\in L^2\}.
\end{equation}
Kenig \cite{Ken} established the global well-posedness of the IVP of
the third order KP-I equation in the following function space
$$Z_0=\Big\{u\in L^2(\mathbb{R}^2):\|u\|_{L^2}+\|\partial_x^{-1}\partial_y u\|_{L^2}
+\|\partial_x^2
u\|_{L^2}+\|\partial_x^{-2}\partial_y^2u\|_{L^2}<\infty\Big\}.$$ As
far as we know, the best well-posedness result of the third KP-I
equation is due to Ionescu, Kenig and Tataru \cite{IonKenTat}. They
set up the globall well posedness of the third order KP-I equation
in the $E(3th)$ space. Thus a more interesting question is to set up
the global well-posedness of the third order KP-I equation in $L^2$.
It is still open.

We now turn our attention back to the fifth order KP-I equation.
Without loss of the generality, we may assume that $\beta=1$ from
now on. The fifth order equation has a higher dispersive term than a
third order KP equation, which helps us to obtain some better
results than the third order KP equation. As before, we first
consider the dispersive function of the fifth order KP equation.
Since there is an interaction between the third order dispersive
term and the fifth order dispersive term, we can not get a
dispersive smoothing effect as (\ref{3KP-1sign}) or
(\ref{3KP-2sign}) for all $(\xi,\mu)\in\mathbb{R}^2$, but we still
have
\begin{equation}\label{5KP-1sign}
|\nabla\omega(\xi,\mu)|=\Big|\Big(5\xi^4+\alpha3\xi^2-\frac{\mu^2}{\xi^2},\,2\frac{\mu}{\xi}\Big)\Big|\gtrsim
|\xi|^2, \text{ if } |\xi|^2>|\alpha|.
\end{equation}
This inequality can help us to recover a full derivative which is
important in the analysis of the fifth order KP-I equation. We also
consider the resonance function
\begin{equation}\label{5KP-Iresonance}
\begin{split}
R(\xi_1,\xi_2,\mu_1,\mu_2)&=\omega(\xi_1+\xi_2,\mu_1+\mu_2)-\omega(\xi_1,\mu_1)-\omega(\xi_2,\mu_2)\\
&=\frac{\xi_1\xi_2}{(\xi_1+\xi_2)}\left((\xi_1+\xi_2)^2\Big[5(\xi_1^2+\xi_1\xi_2+\xi_2^2)-3\alpha\Big]-
\Big(\frac{\mu_1}{\xi_1}-\frac{\mu_2}{\xi_2}\Big)^2\right).
\end{split}
\end{equation}
The first result of the fifth order KP-I equation in the context of
energy space is due to Saut and Tzvetkov \cite{SauTzv2}. They
obtained the locall well-posedness for the fifth order KP-I equation
with data satisfying
$$\|u_0\|_{L^2}+\||D_x|^s u_0\|_{L^2}+\||D_y|^ku_0\|<\infty, \, s\geq 1,\,k\geq 0,\,
|\xi|^{-1}\hat{u_0}(\xi,\mu)\in \mathcal{S}^\prime(\mathbb{R}^2).$$
Here $|D_x|^su_0=(|\xi|^s\hat{u_0})^\vee.$ They also set up the
global well-posedness for the data satisfies $u_0\in L^2$ and
$H(u_0)<\infty$. Recently, Ionescu and Kenig \cite{IonKen} got the
global well-posedness for the IVP of the fifth order periodic KP-I
equation absenting the third order dispersive term with the initial
data in $E(5th)$. For the IVP of the fifth order KP-II equation,
Saut and Tzvetkov \cite{SauTzv2} also obtained the global
well-posedness for the initial data in $L^2$. And they put forward
an open problem whether one can get the local and global
well-posedness of the IVP of the fifth order KP-I equation with the
initial data in $L^2$.

To connect the known results with the $L^2$ conjecture, we introduce
the function space $E_s$ consisting of all the functions satisfying
$$\|f\|_{s}=:\|f\|_{E_s}=\Big\|\Big(1+|\xi|^2+\frac{|\mu|}{|\xi|}\Big)^s
\hat{f}(\xi,\mu)\Big\|_{L^2}<\infty, \forall s\in \mathbb{R}.$$ It
is easy to see when $s=0$, $E_0=L^2$, and when $s=1$, $E_1=E(5th)$.
To get the low regularity of the KP equation, we need a careful
analysis on the time-spatial spaces. In this case, Bourgain type
space is needed. Below, we may abuse $\hat{f}$ as the Fourier
transform of a function in $(x,y)$ or $(x,y,t)$.  One may figure it
out in the context.
\begin{definition}\label{def2}
Let
$\chi_0(\tau-\omega(\xi,\mu))=\chi_{[0,1]}(|\tau-\omega(\xi,\mu)|)$,
$\chi_j(\tau-\omega(\xi,\mu))=\chi_{[2^{j-1},2^{j}]}(|\tau-\omega(\xi,\mu)|)$
for $j\in\mathbb{N}$. For $s,b\in\mathbb{R}$, we define the space
$X_{s,b}$ through the following norm:
\begin{equation} \label{Bourgain}
\|f\|_{X_{s,b}}=\sum_{j\geq
0}2^{jb}\Big\|\chi_j(\tau-\omega(\xi,\mu))(1+|\xi|^2+\frac{|\mu|}{|\xi|})^s\hat{f}(\xi,\mu,\tau)\Big\|_{L^2}.
\end{equation}
\end{definition}

Furthermore, for an interval $I\subset\mathbb{R}$ the localized
Bourgain space $X_{s,b}(I)$ can be defined via requiring
$$
\|u\|_{X_{s,b}(I)}=\inf_{w\in X_{s,b}}\big\{\|w\|_{X_{s,b}}:\quad
w(t)=u(t)\quad \text{on\quad interval}\quad I\big\}.
$$
We now state the well-posedness result in $X_{s,b}$ with initial
data in $E_s$.
\begin{theorem}\label{th:2}
Assume that $\beta=1,\alpha\in\mathbb{R}$, and $1\geq s>0$. For any
real valued function $u_0\in E_s$, there exist $T=T(\|u_0\|_{E_s})$
and a unique solution u of (\ref{5KP}) in $X_{s,\frac{1}{2}+}(I)$
with $I=[-T,T]$. Moreover the map $u_0\rightarrow u$ is smooth from
$E_s$ to $X_{s,\frac{1}{2}+}(I)$. By Sobolev embedding, we have
$u\in C([-T,T];E_s).$ Here $\frac{1}{2}+>\frac{1}{2}$ and is as
close as possible to $\frac{1}{2}$.
\end{theorem}
By (\ref{eq:1.1}) and Theorem \ref{th:2}, we can recover the global
well-posedness of the IVP of the fifth order KP-I equation in the
energy space:
\begin{theorem}\label{global}(see also \cite{SauTzv2})
Assume that $\beta=1$, $\alpha\in\mathbb{R}$, $s=1$. For any real
valued $u_0\in E_1$, there exists a unique solution of the IVP of
the fifth order KP-I equation
$$u\in C(\mathbb{R}, E_1).$$
\end{theorem}

\begin{re} Even though the conjecture that the global well-posedness
for the IVP of the fifth order KP-I equation with data in $L^2$ is
still open, it seems the function space $E_s$ will be expected to
consider this open problem. Since $E_s$ contains the specific
feature $(1+|\xi|^2+|\mu||\xi|^{-1})$ of KP-I equation, and is
different from the Sobolev space $H^{s_1,s_2}$ or $H^s$, we have
independent interest in obtaining the global or local well-posedness
of the IVP of the fifth order KP-I equations in $E_s$ for $s\in
\mathbb R$.

\end{re}

\begin{re}In our argument, dyadic Strichartz estimates are essential. Especially,
when we dispose the ``high-low" interaction in the bilinear
estimate, a low order derivative on the low frequency part is
needed. In this case, $s>0$ is necessary.
\end{re}

Our main argument to prove Theorem \ref{th:2} is to set up a
bilinear estimate as in Section 3 below. Recently, Colliander,
Ionescu, Kenig and Staffilani \cite{ColKenSta3} discovered a
conterexample which showed that one could not set up a similar
bilinear estimate in the Bourgain type space in the third KP-I case.
But we find their conterexample does not work in our case, since the
fifth order dispersive function can help us to recover a full
derivative. Also, we do not recourse to the weighted space. In
\cite{ColKenSta3}, a weighted space is also used to dispose the case
when the very high and very low frequency interaction happens. In
our paper, we can overcome this difficulty by the fifth order
smoothing effect and the dyadic decomposed Strichartz estimate.

In the rest of the paper we would like to use the notation
$A\lesssim B$ if there exists a constant $C>0$ which does not depend
on $B$ such that $A\leq CB.$ If $C<\frac{1}{100}$, we would like to
use $A\ll B$. If there exist $c$ and $C$ which are
$\frac{1}{100}<c<C<100,$ such that $cA\leq B\leq CA$, the notation
$A\sim B$ will be used. And the constants $c$ and $C$ will be
possibly different from line to line.

This paper is organized as follows. In Section 2, we present some
results on linear KP equation and some useful estimates. In Section
3, we present the bilinear estimate which is crucial to set up our
locall well-posedness. In Section 4, we finish the proof of Theorem
\ref{th:2}.

\section{The Linear Estimates}
We begin with the IVP of linear KP equation
\begin{eqnarray}\label{L5KP}
\left\{\begin{array}{ll}
\partial_tu+\alpha\partial_x^3u+\partial_x^5u+\partial_x^{-1}\partial_y^2u=0, \\
u(0,x,y)=u_0(x,y),\quad\quad(x,y)\in\mathbb{R}^2.\end{array}\right.\end{eqnarray}
By Fourier transform, the solution of (\ref{L5KP}) can be defined
as,
$$u=S(t)u_0(x,y)=\int_{\mathbb{R}^2}e^{i(x\xi+y\mu+t\omega(\xi,\mu))}\widehat{u_0}(\xi,\mu)d\xi
d\mu.$$ By Duhamel's formula, (\ref{5KP}) can be reduced to the
integral formulation.
\begin{eqnarray}\label{eq:1}
u(t)=S(t)u_0-\frac 12 \int_0^tS(t-t')\partial_x(u^2(t'))dt'.
\end{eqnarray}
Indeed, to get the local existence result, we apply the fixed point
argument to the nonlinear map defined as the right hand side of the
following integral equation.
\begin{eqnarray}\label{eq:2}
u(t)=\psi(t)\Big[S(t)u_0-\frac{1}{2}
\int_0^tS(t-t')\partial_x(\psi_T^2(t')u^2(t'))dt'\Big],
\end{eqnarray}
where $t\in {\Bbb R}$ and, $\psi$ is a time cut-off function
satisfying
\begin{equation}\label{bump}
\psi\in C_0^\infty({\Bbb R}),\,{\rm supp}\,\psi\subset [-2,\,2],\,
\psi\equiv 1\,{\rm on}\,[-1,\,1],
\end{equation}
 and $\psi_T(\cdot)=\psi(\cdot/T)$.

To run the fixed point argument, we first set up the following
homogeneous and inhomogeneous linear estimates.
\begin{prop}\label{prop1}
Assume $\psi\in C^\infty$ as above and $s\in\mathbb{R}$,
$\frac{1}{2}\leq b< 1$, then
\begin{equation}\label{homo}
\|\psi(t) S(t)u_0\|_{X_{s,b}}\leq C\|u_0\|_{E_s}.
\end{equation}
\begin{equation}\label{inhomo}
\Big\|\psi(t) \int_0^tS(t-t')h(t')dt'\Big\|_{X_{s,b}}\leq
C\|h\|_{X_{s,b-1}}.
\end{equation}
\end{prop}

\begin{proof}
We observe that
\begin{equation}\label{eq:2.2}
(\psi(t)S(t)u_0)\hat{\quad}(\xi,\mu,\tau)=\hat{\psi}(\tau-\omega(\xi,\mu))\widehat{u_0}(\xi,\mu).
\end{equation}
To prove (\ref{homo}), we need to estimate the following integral
expression:
\begin{equation}\label{eq:2.3}
\sum_{j\geq 0}2^{jb}\Big(\int_{\mathbb
R^3}w(\xi,\mu)^{2s}\chi_j(\tau-\omega)|\hat{\psi}(\tau-\omega)|^2|\widehat{u_0}|^2d\xi
d\mu d\tau\Big)^{\frac{1}{2}},
\end{equation}
where $w(\xi,\mu)=(1+|\xi|^2+\frac{|\mu|}{|\xi|})$. We observe that
for $j=0$
\begin{equation}\label{eq:2.4}
\int_{\mathbb
R}|\hat{\psi}(\lambda)|^2\chi_j(\lambda)d\lambda\lesssim\|\hat{\psi}\|^2_{L^\infty}
\end{equation}
and for $j\geq1$
\begin{equation}\label{eq:2.5}
\int_{\mathbb
R}|\hat{\psi}(\lambda)|^2\chi_j(\lambda)d\lambda\lesssim
2^j\frac{1}{(1+2^j)^{2N}}\|(1+|s|)^N\hat{\psi}(s)\|_{L^\infty}^2
\end{equation}
for any $N\in \mathbb N$. When we insert (\ref{eq:2.4}) and
(\ref{eq:2.5}) into (\ref{eq:2.3}) we obtain the bound
\begin{equation}\label{eq:2.6}
\|u_0\|_{E_s}\Big(\|\hat{\psi}\|_{L^\infty}+\sum_{j\geq1}\frac{2^{(\frac{1}{2}+b)j}}{(1+2^j)^{N}}
\|(1+|s|)^N\hat{\psi}(s)\|_{L^\infty}\Big).
\end{equation}
It is easy to see that for $N>2$,
$\displaystyle\sum_{j\geq1}\frac{2^{(\frac{1}{2}+b)j}}{(1+2^j)^N}\leq
C$, then (\ref{homo}) is proved.

To prove (\ref{inhomo}), we write
$$\psi(t)\int_0^tS(t-t')h(t')dt'=I+II,$$
where
$$I=\psi(t)\int_{-\infty}^\infty\int_{\mathbb R^2}e^{i(x\xi+y\mu)}\hat{h}(\xi,\mu,\tau)
\psi(\tau-\omega)\frac{e^{it\tau}-e^{it\omega}}{\tau-\omega(\xi,\mu)}d\xi
d\mu d\tau,$$ and
$$II=\psi(t)\int_{-\infty}^\infty\int_{\mathbb R^2}e^{i(x\xi+y\mu)}\hat{h}(\xi,\mu,\tau)
[1-\psi(\tau-\omega)]\frac{e^{it\tau}-e^{it\omega}}{\tau-\omega(\xi,\mu)}d\xi
d\mu d\tau.$$ By Taylor expansion we can write I as
\begin{equation}\label{eq:2.7}
I=\sum_{k=1}^\infty\frac{i^k}{k!}t^k\psi(t)\int_{\mathbb
R^2}e^{i(x\xi+y\mu+t\omega)}\Big(\int_{-\infty}^\infty
\hat{h}(\xi,\mu,\tau)(\tau-\omega)^{k-1}\psi(\tau-\omega)d\tau\Big)d\xi
d\tau.
\end{equation}
For $k\geq 1$, we write
$$t^k\psi(t)=\psi_k(t).$$
It is easy to show for  $s\in\mathbb R$,
$$|\widehat{\psi_k}(s)|\leq C,$$
and for any $|s|>1$,
$$|\widehat{\psi_k}(s)|\leq C\frac{(1+k)^2}{(1+|s|)^2}.$$
From (\ref{eq:2.7}) it is easy to see
$$I=\sum_{k=1}^\infty \frac{i^k}{k!}\psi_k(t)S(t)h_k(x,y),$$
where
$$\widehat{h_k}(\xi,\mu)=\int_{-\infty}^\infty\hat{h}(\xi,\mu,\tau)(\tau-\omega)^{k-1}\psi(\tau-\omega)d\tau.$$
Then by (\ref{homo}), we obtain
$$\|I\|_{X_{s,b}}\lesssim\sum_{k\geq1}\frac{(1+k)^2}{k!}\|h_k\|_{E_s}.$$
On the other hand, from the definition of $E_s$ and $X_{s,b}$, it is
easy to see that
$$\|h_k\|_{E_s}\lesssim\|h\|_{X_{s,b-1}}.$$
We now pass to $II$. We write $II=II_1+II_2$, where
$$II_1=\psi(t)\int_{-\infty}^\infty\int_{\mathbb
R^2}e^{i(x\xi+y\mu)}\hat{h}(\xi,\mu,\tau)
[1-\psi(\tau-\omega)]\frac{e^{it\tau}}{\tau-\omega(\xi,\mu)}d\xi
d\mu d\tau,$$
$$II_2=-\psi(t)\int_{\mathbb
R^2}e^{i(x\xi+y\mu)}\int_{-\infty}^\infty\hat{h}(\xi,\mu,\tau)
[1-\psi(\tau-\omega)]\frac{e^{it\omega}}{\tau-\omega(\xi,\mu)} d\tau
d\xi d\mu.$$ Again by the definition of $X_{s,b}$, we obtain
$$\|II_1\|_{X_{s,b}}\lesssim\|h\|_{X_{s,b-1}}.$$
By (\ref{homo}), we get
$$\|II_2\|_{X_{s,b}}\lesssim\|\tilde{h}\|_{E_s},$$
where
$$\widehat{\tilde{h}}(\xi,\mu)=\int_{-\infty}^\infty[1-\psi(\tau-\omega)]
\frac{\hat{h}(\xi,\mu,\tau)}{\tau-\omega}d\tau.$$ By the following
estimate
$$\|\tilde{h}\|_{E_s}\lesssim\|h\|_{X_{s,b-1}},$$
we finish the proof of Proposition \ref{prop1}.
\end{proof}

\begin{prop}\cite{BenSau}\label{prop2} Let $\delta(r)=2(\frac{1}{2}-\frac{1}{r}),\,2\leq
r<\infty.$ For any $0<T<1$, there exists $C$ independent of $T$ such
that
\begin{equation}\label{Strichartz}
\||D_x|^{\frac{\delta(r)}{2}}S(t)u_0(x,y)\|_{L^q_T(L^r_{(x,y)})}\leq
C\|u_0\|_{L^2_{(x,y)}},\quad\quad\frac{2}{q}=\delta(r).
\end{equation}
\end{prop}
Here
$$\|f\|_{L^q_T(L^r_{(x,y)})}=\left(\int_{-T}^T
\left(\iint|f(x,y,t)|^rdxdy\right)^{\frac{q}{r}}dt\right)^{\frac{1}{q}}.$$

The following dyadic decomposed Strichartz estimates are crucial in
our bilinear estimates.
\begin{prop}\label{dydicstri}
Let $\chi_j(\xi,\mu,\tau)=\chi_j(\tau-\omega(\xi,\mu)), j\geq 0$,
and $(q,r)$ as in Proposition \ref{prop2}. Denote
$f_j=(\chi_j(\xi,\mu,\tau)|\hat{f}|(\xi,\mu,\tau))^\vee $. For any
$0<T<1$, we have
\begin{equation}\label{eq:3}
\||D_x|^{\frac{\delta(r)}{2}}f_j\|_{L^q_T(L^r_{(x,y)})}\lesssim
2^{\frac{j}{2}}\|f_j\|_{L^2}.
\end{equation}

\end{prop}
Here
$$\|f\|_{L^2}=\left(\iiint|f(\xi,\mu,\tau)|^2d\xi d\mu d\tau\right)^{\frac{1}{2}}.$$
For the sake of convenience, we would like to state the following
special cases.
\begin{equation}\label{stri-2}
\|f_j\|_{L^\infty_T(L^2_{(x,y)})}\lesssim
2^{j/2}\|f_j\|_{L^2(x,y,t)}.
\end{equation}
\begin{equation}\label{stri-4}
\||D_x|^{\frac{1}{4}}f_j\|_{L^4_T(L^4_{(x,y)})}\lesssim
2^{j/2}\|f_j\|_{L^2(x,y,t)}.
\end{equation}
 For
$0<\delta<\frac{1}{2}$
\begin{equation}\label{stri-s-1}
\||D_x|^{\delta}f_j\|_{L^{\frac{1}{\delta}}_T(L^{\frac{2}{1-2\delta}}_{(x,y)})}
\lesssim2^{j/2}\|f_j\|_{L^2_{(x,y,t)}},
\end{equation}
and
\begin{equation}\label{stri-s-2}
\||D_x|^{\frac{1}{2}-\delta}f_j\|_{L^{\frac{2}{1-2\delta}}_T(L^{\frac{1}{\delta}}_{(x,y)})}\lesssim2^{j/2}
\|f_j\|_{L^2_{(x,y,t)}}.
\end{equation}
\begin{proof}: We first note that
$$f_j(x,y,t)=\int_{\mathbb{R}^3}e^{i(x\xi+y\mu+t\tau)}|\hat{f}|\chi_j(\xi,\mu,\tau)d\xi d\mu d\tau.$$
By a simple change of variables we can write
\begin{equation*}
\begin{split}
f_j(x,y,t)&=\int_{\mathbb{R}^3}e^{i(x\xi+y\mu+t(\lambda+\omega))}|\hat{f}|(\xi,\mu,\lambda+\omega)
\chi_j(\lambda)d\xi d\mu d\lambda \\
 &=\int_{\mathbb{R}}e^{it\lambda}\chi_j(\lambda)
\Big[\int_{\mathbb{R}^2}e^{i(x\xi+y\mu+t\omega)}|\hat{f}|(\xi,\mu,\lambda+\omega)
d\xi d\mu \Big]d\lambda \\
&=\int_{\mathbb{R}} e^{it\lambda}\chi_j(\lambda)
S(t)f_\lambda(x,y)d\lambda.
\end{split}
\end{equation*}
Here
$\widehat{f_\lambda}(\xi,\mu)=|\hat{f}|(\xi,\mu,\lambda+\omega)$.
Then (\ref{eq:3}) follows from Minkowski's inequality, Strichartz
estimate (\ref{Strichartz}) and Cauchy-Schwarz inequality.
\end{proof}

To set up the bilinear estimate in the next section, we will
encounter the interaction between high frequency and  very low
frequency. Then the following maximal estimate will be useful when
we dispose the very low frequency.
\begin{prop}\label{prop5}(Maximal estimate)
Let $T_m$ be the operator such that
$\widehat{T_mf}(\xi,\mu,\tau)=m(\xi,\mu)\hat{f}(\xi,\mu,\tau)$. Then
\begin{equation}\label{eq:2.1}
\|T_m(f)\|_{L^2_t(L^\infty_{x,y})}\lesssim
\|m\|_{L^2_{\xi,\mu}}\|f\|_{L^2}.
\end{equation}
\end{prop}

\begin{proof}
We first notice that
$$T_mf(x,y,t)=\int_{\mathbb R^2}\check{m}(x-x',y-y')f(x',y',t)dx'dy'.$$
Here and below, we use $\check{m}$ to denote the inverse Fourier
transform of a function $m$. Then
$$|T_mf(x,y,t)|\lesssim \|m\|_{L^2}\|f(\cdot,\cdot,t)\|_{L^2_{x,y}}.$$
To end the proof one only take the $L^2$ norm in the $t$ variable.
\end{proof}

At the end of this section, we would like to set up the following
proposition, whose idea comes from Lemma 3.1 of \cite{GinTsuVel}

\begin{prop}\label{prop6}
Let $f$ be a function with compact support (in time) in $[-T,T]$ and
$b\geq0$. For any $a>0$, there exists $\sigma=\sigma(a)>0$, such
that
\begin{equation}\label{eq:11}
\|f\|_{X_{0,(b-a)}}\lesssim T^\sigma\|f\|_{X_{0,b}}.
\end{equation}
\end{prop}

\begin{proof}
We first show that
\begin{equation}\label{eq:12}
\|<\tau-\omega>^{-a}\hat{f}\|_{L^2}\lesssim T^\sigma \|f\|_{L^2}.
\end{equation}
We rewrite
$$\|<\tau-\omega>^{-a}\hat{f}\|_{L^2}=\|S(t)<\partial_t>^{-a}S(-t)f\|_{L^2}.$$
Since $S(t)$ is a unit operator in $L^2$ space and preserves the
support properties in time, we have
\begin{equation}
\begin{split}\|<\tau-\omega>^{-a}\hat{f}\|_{L^2}&=
\|<\partial_t>^{-a}S(-t)f\|_{L^2}\lesssim
T^{\frac{1}{2}-\frac{1}{q'}}\|<\partial_t>^{-a}S(-t)f\|_{L^2_{(x,y)}(L^{q'}_t)}\\
&\lesssim T^{\frac{1}{2}-\frac{1}{q'}}\|S(-t)f\|_{L^2},
\end{split}
\end{equation}
 where
$\frac{1}{2}-\frac{1}{q'}=a<\frac{1}{2}$ or $q'=\infty$, if
$a>\frac{1}{2}.$ We now turn to show (\ref{eq:11}) by (\ref{eq:12}).
By the definition of $X_{0,b}$, we have

\begin{equation*}
\begin{split}
\|f\|_{X_{0,b-a}}=&\sum_{j\geq
0}2^{j(b-a)}\|\chi_j(\tau-\omega(\xi,\mu))\hat{f}\|_{L^2}\\
&\lesssim\sum_{j\geq
0}2^{-aj/2}\|<\tau-\omega>^{-a/2}<\tau-\omega>^b\chi_j(\tau-\omega(\xi,\mu))\hat{f}\|_{L^2}\\
&\lesssim\sum_{j\geq
0}2^{-aj/2}T^\sigma\|<\tau-\omega>^b\chi_j(\tau-\omega(\xi,\mu))\hat{f}\|_{L^2}\\
&\lesssim T^\sigma\|f\|_{X_{0,b}}.
\end{split}
\end{equation*}
\end{proof}

\hspace{5mm}

\section{The Bilinear Estimates}

\begin{theorem}\label{biestimate}
Assume $0<s\leq 1$, and $u,v$ with compact time support on $[-T,T]$,
$0<T<1$. There exists $\sigma>0$ such that
\begin{equation}\label{bilinear}\|\partial_x(uv)\|_{X_{s,-\frac{1}{2}+}}\lesssim
T^\sigma\|u\|_{X_{s,\frac{1}{2}+}}\|v\|_{X_{s,\frac{1}{2}+}}.
\end{equation}
Here $-\frac{1}{2}+=(\frac{1}{2}+)-1.$
\end{theorem}
\begin{re}
The bilinear estimate above plays a key role in the method of Picard
iteration. There are many literatures considering the multilinear
estimates. Among them we prefer to pay more attention on
\cite{KenPonVeg} and \cite{Tao}. In \cite{KenPonVeg}, Kenig, Ponce
and Vega present a bilinear estimate in the studying of the IVP of
KdV. It mainly depends on the estimate of the resonance function.
Since in the KdV case, the resonant set is very simple, the
decomposition of frequency method can bring us enough benefit.
Recently, the first two authors \cite{CheLi} obtained the low
regularity of modified
 KdV-Burgers equation by this method. In \cite{Tao},  Tao presented another program to obtain the multilinear estiamtes.
 He used the dual argument and dyadic decomposition to transform the
 multilinear estimate into the estimates of some multipliers on
 some basic boxes. This method can be used to study some more
 complicated cases. We also applied this method in a recent
 paper \cite{CheLiMia} to set up the well-posedness of the IVP of the modified KdV equation with a dissipative term.
As pointed out in \cite{Tao}, the estimate in the box for the KP
equation is much complicated. In this paper, we would like to use
the dyadic decomposition, the Strichartz estimates and the
dispersive smoothing effect to exhaust the structure of the zero set
of KP-I resonance function.
\end{re}
We use the duality to prove the bilinear estimate (\ref{bilinear}).
To make our argument more clear, we would like to divide our
estimates into two catalogs according to the main term in
$(1+|\xi|^2+|\mu||\xi|^{-1})$. It means that we need to estimate,
for $g_j\geq0$,
\begin{equation}\label{eq:3.1}
\begin{split}
\sum_{j\geq0}&2^{j(-\frac{1}{2}+)}\int_{A*}g_j(\xi,\mu,\tau)\chi_j(\tau-\omega(\xi,\mu))\chi_1(\xi,\mu)\\
&\times
|\xi|(1+|\xi|^2)^s|\hat{u}|(\xi_1,\mu_1,\tau_1)|\hat{v}|(\xi_2,\mu_2,\tau_2)d\xi_1d\mu_1d\tau_1d\xi_2d\mu_2d\tau_2,
\end{split}
\end{equation}
and
\begin{equation}\label{eq:3.2}
\begin{split}
\sum_{j\geq0}&2^{j(-\frac{1}{2}+)}\int_{A*}g_j(\xi,\mu,\tau)\chi_j(\tau-\omega(\xi,\mu))\chi_2(\xi,\mu)\\
&\times
|\xi|(1+\frac{|\mu|}{|\xi|})^s|\hat{u}|(\xi_1,\mu_1,\tau_1)|\hat{v}|
(\xi_2,\mu_2,\tau_2)d\xi_1d\mu_1d\tau_1d\xi_2d\mu_2d\tau_2,
\end{split}
\end{equation}
where $A*$ is the set
$\{\xi_1+\xi_2=\xi,\,\mu_1+\mu_2=\mu,\,\tau_1+\tau_2=\tau\}$,
$\chi_1(\xi,\mu)$ is the characteristic function of the set
$\{|\xi|^2\geq\frac{|\mu|}{|\xi|}\}$, $\chi_2(\xi,\mu)$ is the
characteristic function of the set $\{|\xi|^2<\frac{|\mu|}{|\xi|}\}$
and $\|g_j\chi_1\chi_j\|_{L^2}\leq 1$ and
$\|g_j\chi_2\chi_j\|_{L^2}\leq 1$. It is clear that by symmetry one
can always assume that $|\xi_1|\geq|\xi_2|$. The KP-I problem is
difficult since resonant set is complicated. We will decompose $A*$
into several domains. For each domain, we decompose it into some
tiny sets, and use the estimates in Section 2 on these tiny sets.
For instance, when the resonant happens, we will consult to the
maximum estimates and the dyadic decomposed Strichartz estimates. We
start by subdividing $A*$ into three domains of integration by
\begin{description}
  \item[Low-Low interaction domain]
  $$A_1=\{|\xi_1|\geq|\xi_2|;|\xi_1|\leq100\max(1,\sqrt{|\alpha|})\};$$
  \item[High-High interaction domain]
  $$A_2=\{|\xi_1|\geq|\xi_2|;|\xi_2|\sim|\xi_1|\geq100\max(1,\sqrt{|\alpha|})\};$$
  \item[High-Low interaction domain]
  $$A_3=\{|\xi_1|\gg|\xi_2|;|\xi_1|\geq100\max(1,\sqrt{|\alpha|})\}.$$
\end{description}

\begin{proof}{\it of Theorem \ref{biestimate}.} Denote
$$\hat{\phi}_1(\xi,\mu,\tau)=(1+|\xi|^2+|\mu|/|\xi|)^s|\hat{u}|(\xi,\mu,\tau);$$
and
$$\hat{\phi}_2(\xi,\mu,\tau)=(1+|\xi|^2+|\mu|/|\xi|)^s|\hat{v}|(\xi,\mu,\tau).$$
Then we need to prove, there exists $\sigma>0$ such that
$$\|\partial_x(uv)\|_{X_{s,-\frac{1}{2}+}}\lesssim T^\sigma\|\phi_1\|_{X_{0,\frac{1}{2}+}}
\|\phi_2\|_{X_{0,\frac{1}{2}+}}.$$ By Proposition \ref{prop6}, it
suffices to show that
\begin{equation}\label{eq:3.10}
\|\partial_x(uv)\|_{X_{s,-\frac{1}{2}+}}\lesssim
\|\phi_1\|_{X_{0,\frac{1}{2}+}}
\|\phi_2\|_{X_{0,\frac{1}{2}}}+\|\phi_1\|_{X_{0,\frac{1}{2}}}
\|\phi_2\|_{X_{0,\frac{1}{2}+}}.
\end{equation}
We now control the following two terms by the right hand side of
(\ref{eq:3.10}).
\begin{equation}\label{eq:3.3}
\begin{split}
\sum_{j\geq0}&2^{j(-\frac{1}{2}+)}\int_{A*}g_j(\xi,\mu,\tau)\chi_j(\tau-\omega(\xi,\mu))\chi_1(\xi,\mu)\\
&\times
|\xi|(1+|\xi|^2)^s\frac{\hat{\phi}_1(\xi_1,\mu_1,\tau_1)}{(1+|\xi_1|^2+\frac{|\mu_1|}{|\xi_1|})^s}
\frac{\hat{\phi}_2(\xi_2,\mu_2,\tau_2)}{(1+|\xi_2|^2+\frac{|\mu_2|}{|\xi_2|})^s},
\end{split}
\end{equation}
and
\begin{equation}\label{eq:3.4}
\begin{split}
\sum_{j\geq0}&2^{j(-\frac{1}{2}+)}\int_{A*}g_j(\xi,\mu,\tau)\chi_j(\tau-\omega(\xi,\mu))\chi_2(\xi,\mu)\\
&\times
|\xi|(1+\frac{|\mu|}{|\xi|})^s\frac{\hat{\phi}_1(\xi_1,\mu_1,\tau_1)}{(1+|\xi_1|^2+\frac{|\mu_1|}{|\xi_1|})^s}
\frac{\hat{\phi}_2(\xi_2,\mu_2,\tau_2)}{(1+|\xi_2|^2+\frac{|\mu_2|}{|\xi_2|})^s}.
\end{split}
\end{equation}
Another assumption is that function
$$G_{i,j}(x,y,t)=\mathcal{F}^{-1}\Bigg(|\xi|\Big(1+|\xi|^2+\frac{|\mu|}{|\xi|}\Big)^sg_j(\xi,\mu,\tau)
 \chi_j\big(\tau-\omega(\xi,\mu)\big)\chi_i(\xi,\mu)\Bigg)(x,y,t)$$
has compact support in time (supporting in the set $[-T,T]$) for
$i=1,2,j\in\mathbb{N}.$ In fact, if we denote
$$\Phi_i(x,y,t)=\mathcal{F}^{-1}\Bigg(\frac{\hat{\phi}_i(\xi_i,\mu_i,\tau_i)}{\big(1+|\xi_i|^2+\frac{|\mu_i|}
{|\xi_i|}\big)^s}\Bigg)(x,y,t), \,\,\,\,\,\,\text{for}\,\, i=1,2,$$
the
 integral in (\ref{eq:3.3}) and (\ref{eq:3.4}) can be written as
 a inner product
 $<G_{i,j},\Phi_1\Phi_2>$. Since $u$ and $v$ have compact support with respect to $t\in[-T,T]$,
 then $\Phi_1\Phi_2$ has the same compact support in time with $u$ and $v$.
 Thus the inner product $<G_{i,j},\Phi_1\Phi_2>$ can be
 restricted on the interval $[-T,T]$ according to the time axis. It
 means we can assume that $G_{i,j}$ has the same compact support in time.
 We also need some other notations.
 $$\hat{\phi}_{i,j_i}=\hat{\phi}_i\chi_{j_i}(\tau_i-\omega(\xi_i,\mu_i)),\quad\quad
i=1,2,$$
$$\hat{\phi}_{i,j_i,m_i}=\hat{\phi}_i\chi_{j_i}(\tau_i-\omega(\xi_i,\mu_i))\theta_{m_i}(\xi_i),\quad\quad
i=1,2,$$
$$\hat{\phi}_{i,j_i,n_i}=\hat{\phi}_i\chi_{j_i}(\tau_i-\omega(\xi_i,\mu_i))\theta_{n_i}(\mu_i),\quad\quad
i=1,2,$$ and
$$\hat{\phi}_{i,j_i,m_i,n_i}=\hat{\phi}_i\chi_{j_i}(\tau_i-\omega(\xi_i,\mu_i))\theta_{m_i}(\xi_i)
\theta_{n_i}(\mu_i),\quad\quad
i=1,2.$$ Here we used the notation
$\theta_0(\eta)=\chi_{[0,1]}(|\eta|),\,\theta_m(\eta)=\chi_{[2^{m-1},\,2^m]}(|\eta|),\,m\in\mathbb{N}.$
Some times, we may use $g_j$ instead of
$g_j(\xi,\mu,\tau)\chi_j(\tau-\omega(\xi,\mu))$, one can figure out
it in the context. Then we can decompose (\ref{eq:3.3}) and
(\ref{eq:3.4}) by
\begin{equation}\label{1}
\begin{split}
\sum_{j_1,j_2\geq0}\sum_{j\geq0}&2^{j(-\frac{1}{2}+)}
\int_{A*}g_j(\xi,\mu,\tau)\chi_j(\tau-\omega(\xi,\mu))\chi_1(\xi,\mu)\\
&\times
|\xi|(1+|\xi|^2)^s\frac{\hat{\phi}_{1,j_1}(\xi_1,\mu_1,\tau_1)}{(1+|\xi_1|^2+\frac{|\mu_1|}{|\xi_1|})^s}
\frac{\hat{\phi}_{2,j_2}(\xi_2,\mu_2,\tau_2)}{(1+|\xi_2|^2+\frac{|\mu_2|}{|\xi_2|})^s},
\end{split}
\end{equation}
and
\begin{equation}\label{2}
\begin{split}
\sum_{j_1,j_2\geq0}\sum_{j\geq0}&2^{j(-\frac{1}{2}+)}
\int_{A*}g_j(\xi,\mu,\tau)\chi_j(\tau-\omega(\xi,\mu))\chi_2(\xi,\mu)\\
&\times
|\xi|(1+\frac{|\mu|}{|\xi|})^s\frac{\hat{\phi}_{1,j_1}(\xi_1,\mu_1,\tau_1)}{(1+|\xi_1|^2+\frac{|\mu_1|}{|\xi_1|})^s}
\frac{\hat{\phi}_{2,j_2}(\xi_2,\mu_2,\tau_2)}{(1+|\xi_2|^2+\frac{|\mu_2|}{|\xi_2|})^s}.
\end{split}
\end{equation}

\vspace{0.5cm}

\noindent{\bf Low-Low interaction}.

\vspace{0.3cm}

\noindent {\bf Case A:} $|\xi_1+\xi_2|^2\geq
\frac{|\mu_1+\mu_2|}{|\xi_1+\xi_2|}$.

\vspace{0.3cm}

\noindent In this case, we have
$|\xi_1+\xi_2|\lesssim|\xi_1|\lesssim\max(1,\sqrt{|\alpha|})$. And
we also have $|\mu_1+\mu_2|\leq
|\xi_1+\xi_2|^3\lesssim\max(1,|\alpha|^\frac{3}{2}).$ Thus we have
\begin{equation}\label{eq:3.9}
\begin{split}
(\ref{1})&\lesssim\sum_{j,j_1,j_2\geq 0}2^{j(-\frac{1}{2}+)}
\|(m(\xi,\mu)g_j)^\vee\|_{L^2_T(L^\infty_{x,y})}\|\phi_{1,j_1}\|_{L^2(\xi,\mu,\tau)}
\|\phi_{2,j_2}\|_{L^\infty_T(L^2_{x,y})}\\
&\lesssim \sum_{j,j_1,j_2\geq 0}
2^{j(-\frac{1}{2}+)}\|g_j\|_{L^2}\|\phi_{1,j_1}\|_{L^2}2^{j_2/2}\|\phi_{2,j_2}\|_{L^2}\\
&\lesssim
\|\phi_1\|_{X_{0,\frac{1}{2}}}\|\phi_2\|_{X_{0,\frac{1}{2}}}.
\end{split}
\end{equation}
Here
$m(\xi,\mu)=\chi_{|\xi|\lesssim\max(1,|\alpha|^{\frac{1}{2}}),|\mu|\lesssim
\max(1,|\alpha|^{\frac{3}{2}})}$, which belongs to
$L^2(\mathbb{R\times R})$.

\vspace{0.3cm}

\noindent {\bf Case B:}
$|\xi_1+\xi_2|^2\leq\frac{|\mu_1+\mu_2|}{|\xi_1+\xi_2|}$.

\vspace{0.3cm}

\noindent We first note that if
$\frac{|\mu_1+\mu_2|}{|\xi_1+\xi_2|}\leq1$, then argument above can
also bring us the same estimate. We need only to consider the case
$\frac{|\mu_1+\mu_2|}{|\xi_1+\xi_2|}\geq1$.
\begin{equation*}
\begin{split}
(\ref{2})\lesssim\sum_{j_1,j_2\geq0}\sum_{j\geq0}&2^{j(-\frac{1}{2}+)}
\int_{A_1}g_j(\xi,\mu,\tau)\chi_j(\tau-\omega(\xi,\mu))
\chi_2(\xi,\mu)\\
&\times
|\xi|^{1-s}|\mu|^s\frac{\hat{\phi}_{1,j_1}(\xi_1,\mu_1,\tau_1)}{(1+|\xi_1|^2+\frac{|\mu_1|}{|\xi_1|})^s}
\frac{\hat{\phi}_{2,j_2}(\xi_2,\mu_2,\tau_2)}{(1+|\xi_2|^2+\frac{|\mu_2|}{|\xi_2|})^s}.
\end{split}
\end{equation*}
We then consider two subcases.

\vspace{0.3cm}

\noindent {\bf Subcase B1:} $|\mu_1|\leq|\mu_2|$.

\vspace{0.3cm}

\noindent If $\frac{|\mu_2|}{|\xi_2|}\leq |\xi_2|^2$, then
$|\mu_2|\lesssim \max(1,|\alpha|^{\frac{3}{2}}).$ Since
$|\xi_1+\xi_2|\lesssim\max(1,|\alpha|^{1/2})$, we have
\begin{equation}\label{eq:3.41}
\begin{split}
(\ref{2})&\lesssim\sum_{j,j_1,j_2\geq 0}2^{j(-\frac{1}{2}+)}
\|(m(\xi,\mu)g_j)^\vee\|_{L^2_T(L^\infty_{x,y})}\|\phi_{1,j_1}\|_{L^2}
\|\phi_{2,j_2}\|_{L^\infty_T(L^2_{x,y})}\\
&\lesssim \sum_{j,j_1,j_2\geq 0}
2^{j(-\frac{1}{2}+)}\|g_j\|_{L^2}\|\phi_{1,j_1}\|_{L^2}2^{j_2/2}\|\phi_{2,j_2}\|_{L^2}\\
&\lesssim
\|\phi_1\|_{X_{0,\frac{1}{2}}}\|\phi_2\|_{X_{0,\frac{1}{2}}}.
\end{split}
\end{equation}
Here
$m(\xi,\mu)=\chi_{|\xi|\lesssim\max(1,|\alpha|^{\frac{1}{2}}),|\mu|\lesssim
\max(1,|\alpha|^{\frac{3}{2}})}$.

\noindent If $\frac{|\mu_2|}{|\xi_2|}\geq |\xi_2|^2$. We first
consider the case
$\frac{|\mu|}{|\xi|}\lesssim\frac{|\mu_2|}{|\xi_2|}$. Thus we can
choose $\min(\frac{1}{2},s)>\delta>0$ as small as possible such that
\begin{equation*}
\begin{split}
(\ref{2})&\lesssim\sum_{j_1,j_2\geq0}\sum_{j\geq0}2^{j(-\frac{1}{2}+)}
\int_{A_1}g_j(\xi,\mu,\tau)\chi_j(\tau-\omega(\xi,\mu))
\chi_2(\xi,\mu)\\
&\quad\quad\times
|\xi|^{1-\delta}\Big(\frac{|\mu|}{|\xi|}\Big)^{s-\delta}|\mu|^{\delta}\hat{\phi}_{1,j_1}(\xi_1,\mu_1,\tau_1)
\frac{\hat\phi_{2,j_2}(\xi_2,\mu_2,\tau_2)}{\Big(\frac{|\mu_2|}{|\xi_2|}\Big)^{s-\delta}
\Big(\frac{|\mu_2|}{|\xi_2|}\Big)^{\delta}}\\
&\lesssim\sum_{j_1,j_2,j\geq0}
2^{j(-\frac{1}{2}+)}\int_{A_1}g_j(\xi,\mu,\tau)\chi_j(\tau-\omega(\xi,\mu))\chi_2(\xi,\mu)\\
&\quad\quad\times
|\xi|^{1-\delta}\hat{\phi}_{1,j_1}(\xi_1,\mu_1,\tau_1)
|\xi_2|^{\delta}\hat{\phi}_{2,j_2}(\xi_2,\mu_2,\tau_2).
\end{split}
\end{equation*}
\noindent If $j\leq j_2$, by H\"older's inequality and
(\ref{stri-4}), we get
\begin{equation*}
\begin{split}
(\ref{2})&\lesssim \sum_{j_1,j_2\geq0}\sum_{j_2\geq
j\geq0}2^{j(-\frac{1}{2}+)}\||D_x|^{\frac{1}{4}}g_j^{\vee}\|_{L^4_T(L^4_{x,y})}
\||D_x|^\frac{1}{4}\phi_{1,j_1}\|_{L^4_T(L^4_{x,y})}\|\phi_{2,j_2}\|_{L^2}\\
&\lesssim\sum_{j_1,j_2\geq0}\sum_{j_2\geq
j\geq0}2^{j(-\frac{1}{2}+)}2^{j/2}2^{j_1/2}\|g_j\|_{L^2}
\|\phi_{1,j_1}\|_{L^2}\|\phi_{2,j_2}\|_{L^2}\\
&\lesssim\|\phi_1\|_{X_{0,\frac{1}{2}}}\|\phi_2\|_{X_{0,\frac{1}{2}}}.
\end{split}
\end{equation*}

\noindent If $j\geq j_2$, by H\"older's inequality and
(\ref{stri-s-1}) and (\ref{stri-s-2}), we obtain
\begin{equation*}
\begin{split}
(\ref{2})&\lesssim \sum_{j,j_1\geq0}\sum_{j\geq
j_2\geq0}2^{j(-\frac{1}{2}+)}\|g_j\|_{L^2}
\||D_x|^{\frac{1}{2}-\delta}\phi_{1,j_1}\|_{L^\frac{2}{1-2\delta}_T(L^\frac{1}{\delta}_{x,y})}
\||D_x|^\delta\phi_{2,j_2}\|_{L^\frac{1}{\delta}_T(L^\frac{2}{1-2\delta}_{x,y})}\\
&\lesssim\sum_{j,j_1\geq0}\sum_{j\geq
j_2\geq0}2^{j(-\frac{1}{2}+)}2^{j_1/2}2^{j_2/2}\|g_j\|_{L^2}
\|\phi_{1,j_1}\|_{L^2}\|\phi_{2,j_2}\|_{L^2}\\
&\lesssim\|\phi_1\|_{X_{0,\frac{1}{2}}}\|\phi_2\|_{X_{0,\frac{1}{2}}}.
\end{split}
\end{equation*}
If $\frac{|\mu|}{|\xi|}\gg\frac{|\mu_2|}{|\xi_2|}$ and
$0<s\leq\frac{1}{2}$, the proof above can also work. We only need to
estimate the case $\frac{1}{2}<s\leq 1.$
\begin{equation*}
\begin{split}
(\ref{2})&\lesssim\sum_{j_1,j_2,j\geq0}
2^{j(-\frac{1}{2}+)}\int_{A_1}g_j(\xi,\mu,\tau)\chi_j(\tau-\omega(\xi,\mu))\chi_2(\xi,\mu)\\
&\quad\quad\times
|\xi|^{1-s}|\mu|^s\frac{\hat{\phi}_{1,j_1}(\xi_1,\mu_1,\tau_1)}{(1+|\mu_1|)^s}
\frac{\hat{\phi}_{2,j_2}(\xi_2,\mu_2,\tau_2)}{|\mu_2|^{s}}.
\end{split}
\end{equation*}
In addition, we decompose $|\mu_1|\sim 2^{n_1}$ for $n_1\geq0$. Thus
\begin{equation*}
\begin{split}
(\ref{2})&\lesssim \sum_{j_1,j_2\geq0}\sum_{j\geq0}\sum_{n_1\geq
0}2^{j(-\frac{1}{2}+)}2^{-n_1s}\|g_j\|_{L^2}
\|\phi_{1,j_1,n_1}\|_{L^2_T(L^\infty_{x,y})}\|\phi_{2,j_2}\|_{L^\infty_T(L^2_{x,y})}\\
&\lesssim\sum_{j_1,j_2\geq0}\sum_{j\geq0}\sum_{n_1\geq
0}2^{j(-\frac{1}{2}+)}2^{j_2/2}2^{-n_1(s-\frac{1}{2})}\|g_j\|_{L^2}
\|\phi_{1,j_1,n_1}\|_{L^2}\|\phi_{2,j_2}\|_{L^2}\\
&\lesssim\|\phi_1\|_{X_{0,\frac{1}{2}}}\|\phi_2\|_{X_{0,\frac{1}{2}}}.
\end{split}
\end{equation*}
Here we used the fact that $|\xi_1|\lesssim\max(1,\sqrt{|\alpha|})$
and $|\mu_1|\lesssim 2^{n_1}$ with Proposition \ref{prop5}.

\vspace{0.3cm}

\noindent {\bf Subcase B2:} $|\mu_1|\geq|\mu_2|$.

\noindent If $|\mu_2|<1$, we obtain
\begin{equation*}
\begin{split}
(\ref{2})&\lesssim\sum_{j,j_1,j_2\geq 0}2^{j(-\frac{1}{2}+)}
\|g_j\|_{L^2}\|\phi_{1,j_1}\|_{L^\infty_T(L^2_{x,y})}
\|(m(\xi_2,\mu_2)\hat{\phi}_{2,j_2})^\vee\|_{L^2_T(L^\infty_{x,y})}\\
&\lesssim \sum_{j,j_1,j_2\geq 0}
2^{j(-\frac{1}{2}+)}\|g_j\|_{L^2}2^{j_1/2}\|\phi_{1,j_1}\|_{L^2}\|\phi_{2,j_2}\|_{L^2}\\
&\lesssim
\|\phi_1\|_{X_{0,\frac{1}{2}}}\|\phi_2\|_{X_{0,\frac{1}{2}}}.
\end{split}
\end{equation*}
Here $m(\xi_2,\mu_2)$ denotes the characteristic function of the set
$\{(\xi_2,\mu_2); |\xi_2|\lesssim \max(1,|\alpha|^{\frac{1}{2}}),$ $
|\mu_2|<1\}$. Thus, we need only to consider the case $|\mu_2|\geq
1$. In this case, we can run the same argument with the Subcase B1
by interchanging the positions of $|\mu_1|$ and $|\mu_2|$. We omit
the details.

\vspace{0.5cm}

\noindent {\bf High-High interaction}

\vspace{0.3cm}

\noindent {\bf Case A:}
$|\xi_1+\xi_2|^2\geq\frac{|\mu_1+\mu_2|}{|\xi_1+\xi_2|}.$

\vspace{0.3cm}

\noindent We can also assume that $|\xi_1+\xi_2|\gtrsim
\max(1,|\alpha|^{1/2}).$ Otherwise we go back to (\ref{eq:3.9}). We
now run dyadic decomposition with respect to $|\xi_1|\sim 2^{m_1}$
(hence $|\xi_2|\sim 2^{m_1}$) and $|\xi|\sim 2^m$ with $m_1+1\geq
m\geq0$.
\begin{equation*}
\begin{split}
(\ref{1})\lesssim\sum_{j_1,j_2\geq0}\sum_{j\geq0}\sum_{m_1+1\geq
m\geq0}&2^{j(-\frac{1}{2}+)}
\int_{A_2}g_j(\xi,\mu,\tau)\chi_j(\tau-\omega(\xi,\mu))\chi_1(\xi,\mu)\\
&\times 2^{m(1+2s)}2^{m_1(-4s)}
\hat{\phi}_{1,j_1,m_1}(\xi_1,\mu_1,\tau_1)
\hat{\phi}_{2,j_2,m_1}(\xi_2,\mu_2,\tau_2).
\end{split}
\end{equation*}
We now consider two subcases.

\vspace{0.3cm}

\noindent {\bf Subcase A1:} $\max(j,j_2)\geq 2m_1.$

\vspace{0.3cm}

\noindent If $j\leq j_2$, we obtain
\begin{equation*}
\begin{split}
(\ref{1})&\lesssim \sum_{j_1,j_2\geq0}\sum_{j_2\geq
j\geq0}\sum_{\frac{j_2}{2}\geq m_1>
0}2^{j(-\frac{1}{2}+)}2^{m_1(\frac{1}{2}-2s)}\||D_x|^{\frac{1}{4}}g_j^\vee\|_{L^4_T(L^4_{x,y})}
\||D_x|^\frac{1}{4}\phi_{1,j_1,m_1}\|_{L^4_T(L^4_{x,y})}\\
&\quad\quad\quad\quad\times\|\phi_{2,j_2,m_1}\|_{L^2}\\
&\lesssim\sum_{j_1,j_2\geq0}\sum_{j_2\geq
j\geq0}2^{j(-\frac{1}{2}+)}2^{j/2}2^{j_1/2}
2^{j_2(\frac{1}{4}-s)}\|g_j\|_{L^2}
\|\phi_{1,j_1}\|_{L^2}\|\phi_{2,j_2}\|_{L^2}\\
&\lesssim\sum_{j_1,j_2\geq0}\sum_{j_2\geq
j\geq0}2^{j(-\frac{1}{2}+)}2^{j/2}2^{j_1/2}
2^{j_2(\frac{1}{2}+)}2^{-j_2(s+\frac{1}{4}+)}\|g_j\|_{L^2}
\|\phi_{1,j_1}\|_{L^2}\|\phi_{2,j_2}\|_{L^2}\\
&\lesssim\|\phi_1\|_{X_{0,\frac{1}{2}}}\|\phi_2\|_{X_{0,\frac{1}{2}+}}.
\end{split}
\end{equation*}

\noindent If $j>j_2$, we also have
\begin{equation*}
\begin{split}
(\ref{1})&\lesssim
\sum_{j_1\geq0}\sum_{j>j_2\geq0}\sum_{\frac{j}{2}\geq m_1>
0}2^{j(-\frac{1}{2}+)}2^{m_1(\frac{1}{2}-2s)}\||D_x|^{\frac{1}{4}}\phi_{2,j_2,m_1}\|_{L^4_T(L^4_{x,y})}
\||D_x|^\frac{1}{4}\phi_{1,j_1,m_1}\|_{L^4_T(L^4_{x,y})}\\
&\quad\quad\quad\quad\times\|g_j\|_{L^2}\\
&\lesssim\sum_{j_1\geq0}\sum_{j>j_2\geq0}2^{j(-\frac{1}{2}+)}
2^{j\max(0,\frac{1}{4}-s)}2^{j_1/2}2^{j_2/2}\|g_j\|_{L^2}
\|\phi_{1,j_1}\|_{L^2}\|\phi_{2,j_2}\|_{L^2}\\
&\lesssim\|\phi_1\|_{X_{0,\frac{1}{2}}}\|\phi_2\|_{X_{0,\frac{1}{2}}}.
\end{split}
\end{equation*}

\vspace{0.3cm}

\noindent {\bf Subcase A2:} $\max(j,j_2)\leq 2m_1$.

\vspace{0.3cm}

\noindent {\bf Subsubcase 1:}
$\Big|\frac{\mu_1}{\xi_1}-\frac{\mu_2}{\xi_2}\Big|^2\leq
\frac{1}{2}|\xi_1+\xi_2|^2|5(\xi_1^2+\xi_1\xi_2+\xi_2^2)-3\alpha|.$

\vspace{0.3cm}

\noindent In this case, the resonant interaction does not happen. By
the definition of resonance function, we can get a useful estimate.
Writing
\begin{multline}\label{5KP-Iresonance}
\tau-\omega(\xi_1+\xi_2,\mu_1+\mu_2)-\tau_1+\omega(\xi_1,\mu_1)-\tau_2+\omega(\xi_2,\mu_2)\\
=-\frac{\xi_1\xi_2}{(\xi_1+\xi_2)}\left((\xi_1+\xi_2)^2\Big[5(\xi_1^2+\xi_1\xi_2+\xi_2^2)-3\alpha\Big]-
\Big(\frac{\mu_1}{\xi_1}-\frac{\mu_2}{\xi_2}\Big)^2\right),
\end{multline}
since $<\tau-\omega(\xi,\mu)>\sim 2^j$,
$<\tau_1-\omega(\xi_1,\mu_1)>\sim 2^{j_1}$ and
$<\tau_2-\omega(\xi_2,\mu_2)>\sim 2^{j_2}$, we have
$2^{\max(j,j_1,j_2)}\geq |\xi_1|^4|\xi_1+\xi_2|\geq |\xi_1|^4\sim
2^{4m_1}.$ It is clear that we have $j_1=\max(j,j_1,j_2)\geq 4m_1.$
Thus $|\xi_1+\xi_2|\lesssim 2^{j_1-4m_1}.$ We now choose $\delta>0$
such that $\min(\frac{1}{4},s)>\delta>0$ and
$1-4\delta+\frac{1}{2}>\frac{1}{2}+$. Therefor
\begin{equation*}
\begin{split}
(\ref{1})&\lesssim \sum_{j_1\geq0}\sum_{\frac{j_1}{4}\geq m_1\geq
0}\sum_{2m_1\geq
\max(j_2,j)\geq0}2^{j(-\frac{1}{2}+)}2^{(j_1-4m_1)(\frac{1}{2}-2\delta)}
\||D_x|^{\frac{1}{4}}g_j^\vee\|_{L^4_T(L^4_{x,y})}
\|\phi_{1,j_1,m_1}\|_{L^2_t(L^2_{x,y})}\\
&\quad\quad\quad\quad\times\||D_x|^\frac{1}{4}\phi_{2,j_2,m_1}\|_{L^4_T(L^4_{x,y})}\\
&\lesssim\sum_{j_1\geq0}\sum_{m_1\geq 0}\sum_{2m_1\geq
\max(j_2,j)\geq0}2^{j(-\frac{1}{2}+)}2^{j/2}2^{(j_1-4m_1)(\frac{1}{2}-2\delta)}2^{j_2/2}\|g_j\|_{L^2}
\|\phi_{1,j_1,m_1}\|_{L^2}\|\phi_{2,j_2,m_1}\|_{L^2}\\
&\lesssim\|\phi_1\|_{X_{0,\frac{1}{2}}}\|\phi_2\|_{X_{0,\frac{1}{2}}}.
\end{split}
\end{equation*}

\vspace{0.3cm}

\noindent {\bf Subsubcase 2:}
$\Big|\frac{\mu_1}{\xi_1}-\frac{\mu_2}{\xi_2}\Big|^2>
\frac{1}{2}|\xi_1+\xi_2|^2|5(\xi_1^2+\xi_1\xi_2+\xi_2^2)-3\alpha|.$

\vspace{0.3cm}

\noindent In this case, the resonant interaction may happen. We have
to do some delicate estimates. Let
$\theta_1=\tau_1-\omega(\xi_1,\mu_1)$ and
$\theta_2=\tau_2-\omega(\xi_2,\mu_2)$, we can control (\ref{1}) by
\begin{equation}\label{eq:3.14}
\begin{split}
&\sum_{j_1,m_1\geq0}\sum_{2m_1>\max(j,j_2)\geq0}2^{j(-\frac{1}{2}+)}2^{m_1(1-2s)}\int
g_j(\xi,\mu,\theta_1+\omega(\xi_1,\mu_1)+\theta_2+\omega(\xi_2+\mu_2))
\\
&\times\chi_j(\theta_1+\theta_2+\omega(\xi_1,\mu_2)+\omega(\xi_2+\mu_2)-
\omega(\xi_1+\xi_2,\mu_1+\mu_2))
\\
&\times\hat{\phi}_{1,j_1,m_1}(\xi_1,\mu_1,\theta_1+\omega(\xi_1,\mu_1))
\hat{\phi}_{2,j_2,m_1}(\xi_2,\mu_2,\theta_2+\omega(\xi_2,\tau_2))d\xi_1d\mu_1d\xi_2d\mu_2d\theta_1d\theta_2.
\end{split}
\end{equation}

\noindent We divide above quantity into two cases.

\vspace{0.3cm}

\noindent {\bf Subsubsubcase 2a:}
$\Big|5(\xi_1^4-\xi_2^4)-3\alpha(\xi_1^2-\xi_2^2)-
\Big[\Big(\frac{\mu_1}{\xi_1}\Big)^2-\Big(\frac{\mu_2}{\xi_2}\Big)^2\Big]\Big|>1.$

\vspace{0.3cm}

\noindent We change the variables by
\begin{equation} \label{eq:3.5}
\left\{ \begin{aligned}
        u&=\xi_1+\xi_2 \\
        v&=\mu_1+\mu_2 \\
        w&=\theta_1+\omega(\xi_1,\mu_1)+\theta_2+\omega(\xi_2+\mu_2)\\
        \mu_2&=\mu_2.
        \end{aligned} \right.
\end{equation}
The determinant of the Jacobian associating to this change of
variables is
\begin{equation}
\begin{split}
J_{\mu}&=\left|\begin{array}{cccc}
          1 & 1 & 0 & 0 \\
          \,& \,& \,& \,\\
          0 & 0 & 1 & 1 \\
          \,& \,& \,& \,\\
          5\xi^4_1-3\alpha\xi_1^2-\frac{\mu_1^2}{\xi_1^2} &
5\xi^4_2-3\alpha\xi_2^2-\frac{\mu_2^2}{\xi_2^2} & 2\frac{\mu_1}{\xi_1} & 2\frac{\mu_2}{\xi_2} \\
          \,& \,& \,& \,\\
          0 & 0 & 0 & 1
        \end{array}\right|\\
&=5(\xi_1^4-\xi_2^4)-3\alpha(\xi_1^2-\xi_2^2)-
\Big[\Big(\frac{\mu_1}{\xi_1}\Big)^2-\Big(\frac{\mu_2}{\xi_2}\Big)^2\Big]
\end{split}\end{equation}
Thus $|J_\mu|>1.$ We have
\begin{equation}\label{eq:3.6}
\begin{split}
(\ref{1})&\lesssim\sum_{j_1,m_1\geq0}\sum_{2m_1>\max(j,j_2)\geq0}2^{j(-\frac{1}{2}+)}2^{m_1(1-2s)}\int
g_j\chi_j(u,v,w)\\
&
\times|J_\mu|^{-1}H(u,v,w,\mu_2,\theta_1,\theta_2)dudvdwd\mu_2d\theta_1d\theta_2.
\end{split}
\end{equation}
Here $H(u,v,w,\mu_2,\theta_1,\theta_2)$ denotes the transformation
of $\hat\phi_{1,j_1,m_1}\hat\phi_{2,j_2,m_1}$. For fixed
$\theta_1,\theta_2,$ $\xi_1,\xi_2,\mu_1$, we calculate the set
length where the free variable $\mu_2$ can range. More precisely, we
denote this length by $\Delta_{\mu_2}$. Let
$$f(\mu)=\theta_1+\theta_2-\frac{\xi_1\xi_2}{(\xi_1+\xi_2)}\left((\xi_1+\xi_2)^2
\Big[5(\xi_1^2+\xi_1\xi_2+\xi_2^2)-3\alpha\Big]-
\Big(\frac{\mu_1}{\xi_1}-\frac{\mu}{\xi_2}\Big)^2\right),$$ we have
$|f'(\mu_2)|>|\xi_1|^2$. Since
\begin{equation}\label{eq:3.11}\begin{split}&|\theta_1+\theta_2+\omega(\xi_1,\mu_2)+\omega(\xi_2+\mu_2)-
\omega(\xi_1+\xi_2,\mu_1+\mu_2)|\\
=&\left|\theta_1+\theta_2-\frac{\xi_1\xi_2}{(\xi_1+\xi_2)}\left((\xi_1+\xi_2)^2
\Big[5(\xi_1^2+\xi_1\xi_2+\xi_2^2)-3\alpha\Big]-
\Big(\frac{\mu_1}{\xi_1}-\frac{\mu_2}{\xi_2}\Big)^2\right)\right|\sim
2^j,\end{split}\end{equation}  This means that we have
$\Delta_{\mu_2}\leq 2^{j-2m_1}$. By Cauchy-Schwarz and the inverse
change of variables we have

\begin{equation*}
\begin{split}
&\int g_j\chi_j(u,v,w)
|J_\mu|^{-1}H(u,v,w,\mu_2,\theta_1,\theta_2)dudvdwd\mu_2d\theta_1d\theta_2\\
\lesssim&2^{j/2-m_1}\int g_j\chi_j(u,v,w)
\left(\int|J_\mu|^{-2}H^2(u,v,w,\mu_2,\theta_1,\theta_2)d\mu_2\right)^{1/2}dudvdwd\theta_1d\theta_2\\
\lesssim&2^{j/2-m_1}\|g_j\chi_j\|_{L^2}\int
\left(\int|J_\mu|^{-2}H^2(u,v,w,\mu_2,\theta_1,\theta_2)dudvdwd\mu_2\right)^{1/2}d\theta_1d\theta_2\\
\lesssim&2^{j/2-m_1}\|g_j\chi_j\|_{L^2}\int
\left(\int|J_\mu|^{-1}H^2(u,v,w,\mu_2,\theta_1,\theta_2)dudvdwd\mu_2\right)^{1/2}d\theta_1d\theta_2\\
=&2^{j/2-m_1}\|g_j\chi_j\|_{L^2}\int
\Big(\int\prod_{i=1,2}\hat\phi^2_{i,j_i,m_1}(\xi_i,\mu_i,\theta_i+\omega(\xi_i,\mu_i))
d\xi_1d\mu_1d\xi_2d\mu_2\Big)^{1/2}d\theta_1d\theta_2\\
\lesssim&2^{j/2-m_1}2^{j_1/2}2^{j_2/2}\|g_j\|_{L^2}\|\phi_{1,j_1,m_1}\|_{L^2}\|\phi_{2,j_2,m_1}\|_{L^2}.
\end{split}
\end{equation*}
It follows from (\ref{eq:3.6}) that
\begin{equation*}
\begin{split}
(\ref{1})&\lesssim
\sum_{m_1,j_1\geq0}\sum_{2m_1>\max(j_2,j)\geq0}2^{j(-\frac{1}{2}+)}2^{j/2-m_1}2^{m_1(1-2s)}
2^{j_1/2}2^{j_2/2}\|g_j\|_{L^2}
\|\phi_{1,j_1,m_1}\|_{L^2}\|\phi_{2,j_2,m_1}\|_{L^2}.\\
&\lesssim\|\phi_1\|_{X_{0,\frac{1}{2}}}\|\phi_2\|_{X_{0,\frac{1}{2}}}.
\end{split}
\end{equation*}

\vspace{0.3cm}

\noindent {\bf Subsubsubcase 2b:}
$\Big|5(\xi_1^4-\xi_2^4)-3\alpha(\xi_1^2-\xi_2^2)-
\Big[\Big(\frac{\mu_1}{\xi_1}\Big)^2-\Big(\frac{\mu_2}{\xi_2}\Big)^2\Big]\Big|\leq1.$

\vspace{0.3cm}

\noindent In this case the change of variables above cannot be used
because the determinant of Jacobian may become zero. We consider
 the change of variables instead:
\begin{equation} \label{eq:3.7}
\left\{ \begin{aligned}
        u&=\xi_1+\xi_2 \\
        v&=\mu_1+\mu_2 \\
        w&=\theta_1+\omega(\xi_1,\mu_1)+\theta_2+\omega(\xi_2+\mu_2)\\
        \xi_1&=\xi_1.
        \end{aligned} \right.
\end{equation}
In this case the determinant of Jacobian $J_\xi$ is given by
\begin{equation}
\begin{split}
J_{\xi}&=\left|\begin{array}{cccc}
          1 & 1 & 0 & 0 \\
          \,& \,& \,& \,\\
          0 & 0 & 1 & 1 \\
          \,& \,& \,& \,\\
          5\xi^4_1-3\alpha\xi_1^2-\frac{\mu_1^2}{\xi_1^2} &
5\xi^4_2-3\alpha\xi_2^2-\frac{\mu_2^2}{\xi_2^2} & 2\frac{\mu_1}{\xi_1} & 2\frac{\mu_2}{\xi_2} \\
          \,& \,& \,& \,\\
          1 & 0 & 0 & 0
        \end{array}\right|\\
&=2\Big(\frac{\mu_1}{\xi_1}-\frac{\mu_2}{\xi_2}\Big).
\end{split}\end{equation}
An easy calculation shows that $|J_\xi|\gtrsim|\xi_1|.$ In this
time, we fixed $\theta_1,\theta_2,\xi_2,\mu_1,\mu_2$, and calculate
the interval length $\Delta_{\xi_1}$ of the free variable $\xi_1$.
Set
\begin{equation}\label{eq:3.18}h(\xi)=5(\xi^4-\xi_2^4)-3\alpha(\xi^2-\xi_2^2)-\Big[\Big(\frac{\mu_1}{\xi}\Big)^2-
\Big(\frac{\mu_2}{\xi_2}\Big)^2\Big].
\end{equation}
We compute
\begin{equation}\label{eq:3.17}
h'(\xi)=20\xi^3-6\alpha\xi+2(\mu_1/\xi)^2\xi^{-1}.
\end{equation}
Since now $h'(\xi_1)$ has the same sign as $\xi_1$, we have
$|h'(\xi_1)|\gtrsim|\xi_1|^3$. Thus $\Delta_{\xi_1}\lesssim
2^{-3m_1}.$ Remind
\begin{equation}\label{eq:3.8}
\begin{split}
(\ref{1})&\lesssim\sum_{j_1,m_1\geq0}\sum_{2m_1>\max(j,j_2)\geq0}2^{j(-\frac{1}{2}+)}2^{m_1(1-2s)}\int
g_j\chi_j(u,v,w)\\
&
\times|J_\xi|^{-1}H(u,v,w,\xi_1,\theta_1,\theta_2)dudvdwd\xi_1d\theta_1d\theta_2.
\end{split}
\end{equation}
Again denote by $H(u,v,w,\xi_1,\theta_1,\theta_2)$ the
transformation of $\prod_{i=1,2}\hat\phi_{i,j_i,m_1}$ under the
change of variables (\ref{eq:3.7}).
\begin{equation*}
\begin{split}
&\int g_j\chi_j(u,v,w)
|J_\xi|^{-1}H(u,v,w,\xi_1,\theta_1,\theta_2)dudvdwd\xi_1d\theta_1d\theta_2\\
\lesssim&2^{-\frac{3}{2}m_1}\int g_j\chi_j(u,v,w)
\left(\int|J_\xi|^{-2}H^2(u,v,w,\xi_1,\theta_1,\theta_2)d\xi_1\right)^{1/2}dudvdwd\theta_1d\theta_2\\
\lesssim&2^{-\frac{3}{2}m_1}\|g_j\chi_j\|_{L^2}\int
\left(\int|J_\xi|^{-2}H^2(u,v,w,\xi_1,\theta_1,\theta_2)dudvdwd\xi_1\right)^{1/2}d\theta_1d\theta_2\\
\lesssim&2^{-2m_1}\|g_j\chi_j\|_{L^2}\int
\left(\int|J_\xi|^{-1}H^2(u,v,w,\xi_1,\theta_1,\theta_2)dudvdwd\xi_1\right)^{1/2}d\theta_1d\theta_2\\
=&2^{-2m_1}\|g_j\chi_j\|_{L^2}\int
\Big(\int\prod_{i=1,2}\hat\phi^2_{i,j_i,m_1}(\xi_i,\mu_i,\theta_i+\omega(\xi_i,\mu_i))
d\xi_id\mu_i\Big)^{1/2}d\theta_1d\theta_2\\
\lesssim&2^{-2m_1}2^{j_1/2}2^{j_2/2}\|g_j\|_{L^2}\|\phi_{1,j_1,m_1}\|_{L^2}\|\phi_{2,j_2,m_1}\|_{L^2}.
\end{split}
\end{equation*}
Thus
\begin{equation*}
\begin{split}
(\ref{1})&\lesssim
\sum_{m_1,j_1\geq0}\sum_{2m_1>\max(j_2,j)\geq0}2^{j(-\frac{1}{2}+)}2^{-2m_1}2^{m_1(1-2s)}2^{j_1/2}2^{j_2/2}\|g_j\|_{L^2}
\|\phi_{1,j_1,m_1}\|_{L^2}\|\phi_{2,j_2,m_1}\|_{L^2}.\\
&\lesssim\|\phi_1\|_{X_{0,\frac{1}{2}}}\|\phi_2\|_{X_{0,\frac{1}{2}}}.
\end{split}
\end{equation*}

\vspace{0.3cm}

\noindent {\bf Case B:}
$|\xi_1+\xi_2|^2\leq\frac{|\mu_1+\mu_2|}{|\xi_1+\xi_2|}.$

\vspace{0.3cm}

\noindent If $\frac{|\mu_1+\mu_2|}{|\xi_1+\xi_2|}\lesssim 1$, this
case can also be proved by (\ref{eq:3.9}). Thus we need only to
consider the case $\frac{|\mu_1+\mu_2|}{|\xi_1+\xi_2|}\gtrsim 1$.

\vspace{0.3cm}

\noindent {\bf Subcase B1:} $|\mu_1|\leq|\mu_2|$.

\vspace{0.3cm}

\noindent {\bf Subsubcase B1a:} $\frac{|\mu_2|}{|\xi_2|}\leq
|\xi_2|^2$.

In this case, $|\mu_2|\leq |\xi_2|^3$ and $|\mu_1+\mu_2|\leq
2|\xi_2|^3$. We now decompose $|\xi_1|\sim|\xi_2|\sim 2^{m_2}$. Then
in this case we bound (\ref{2}) by
\begin{equation}\label{eq:3.12}
\begin{split}
\sum_{j_1,j_2,j\geq0}\sum_{m_2\geq0}&2^{j(-\frac{1}{2}+)}
\int_{A_2}g_j(\xi,\mu,\tau)\chi_j(\tau-\omega(\xi,\mu))\chi_2(\xi,\mu)\\
&\times |\xi_1+\xi_2|^{1-s}2^{-m_2s}
\hat\phi_{1,j_1,m_2}(\xi_1,\mu_1,\tau_1)
\hat\phi_{2,j_2,m_2}(\xi_2,\mu_2,\tau_2).
\end{split}
\end{equation}

We first consider the case that two high frequency waves interaction
forms a very low wave i.e. $|\xi_1+\xi_2|<1$.
\begin{equation*}
\begin{split}
(\ref{2})&\lesssim \sum_{j,m_2\geq0}\sum_{
j_1,j_2\geq0}2^{j(-\frac{1}{2}+)}2^{m_2(-\frac{1}{2}-s)}\||D_x|^{\frac{1}{4}}\phi_{2,j_2,m_2}\|_{L^4_T(L^4_{x,y})}
\||D_x|^\frac{1}{4}\phi_{1,j_1,m_2}\|_{L^4_T(L^4_{x,y})}\\
&\quad\quad\quad\quad\times\|g_j\|_{L^2}\\
&\lesssim\sum_{j,m_2\geq0}\sum_{
j_1,j_2\geq0}2^{j(-\frac{1}{2}+)}2^{m_2(-\frac{1}{2}-s)}2^{j_1/2}2^{j_2/2}\|g_j\|_{L^2}
\|\phi_{1,j_1,m_2}\|_{L^2}\|\phi_{2,j_2,m_2}\|_{L^2}\\
&\lesssim\|\phi_1\|_{X_{0, \frac{1}{2}}}\|\phi_2\|_{X_{0,
\frac{1}{2}}}.
\end{split}
\end{equation*}
For the case $|\xi_1+\xi_2|>1$, one can use the argument in case A
again to obtain
\begin{equation*}
(\ref{2})\lesssim\|\phi_1\|_{X_{0,\frac{1}{2}+}}\|\phi_2\|_{X_{0,\frac{1}{2}}}
+\|\phi_1\|_{X_{0,\frac{1}{2}}}\|\phi_2\|_{X_{0,\frac{1}{2}+}}.
\end{equation*}

\vspace{0.3cm}

\noindent {\bf Subsubcase B1b:}  $\frac{|\mu_2|}{|\xi_2|}\geq
|\xi_2|^2$.

\vspace{0.3cm}

We bound (\ref{2}) by
\begin{equation*}
\begin{split}
\sum_{j_1,j_2,j\geq0}&2^{j(-\frac{1}{2}+)}
\int_{A_2}g_j(\xi,\mu,\tau)\chi_j(\tau-\omega(\xi,\mu))\chi_2(\xi,\mu)\\
&\times |\xi_1+\xi_2|^{1-s}|\mu_1+\mu_2|^s
\frac{\hat\phi_{1,j_1}(\xi_1,\mu_1,\tau_1)}{|\xi_1|^{2s}}
\frac{\hat\phi_{2,j_2}(\xi_2,\mu_2,\tau_2)}{\Big(\frac{|\mu_2|}{|\xi_2|}\Big)^s}.
\end{split}
\end{equation*}
Of course a dyadic decomposition with respect to $\xi_1$ is also
needed. Let $|\xi_1|\sim 2^{m_1}$, we bound (\ref{2}) by
\begin{equation*}
\begin{split}
\sum_{j_1,j_2,j\geq0}\sum_{m_1\geq0}&2^{j(-\frac{1}{2}+)}
\int_{A_2}g_j(\xi,\mu,\tau)\chi_j(\tau-\omega(\xi,\mu))\chi_2(\xi,\mu)\\
&\times |\xi_1+\xi_2|^{1-s}2^{-m_1s}
\hat\phi_{1,j_1,m_1}(\xi_1,\mu_1,\tau_1)
\hat\phi_{2,j_2}(\xi_2,\mu_2,\tau_2).
\end{split}
\end{equation*}
Then one can also run the above argument by considering two cases:
$|\xi_1+ \xi_2|\leq 1$ and $|\xi_1+\xi_2|\geq 1.$ We now give some
details in the case $|\xi_1+\xi_2|\geq 1$.

\vspace{0.3cm}

\noindent {\bf Subsubsubcase 1:} $\max(j,j_2)\geq 2m_1.$

\vspace{0.3cm}

\noindent If $j\leq j_2$ and $0<s\leq\frac{1}{4}$, we choose
$0<\delta<\frac{1}{2}$
\begin{equation*}
\begin{split}
(\ref{2})&\lesssim\sum_{j_1,j_2\geq0}\sum_{j_2\geq
\max(j,2m_1)\geq0}2^{j(-\frac{1}{2}+)}2^{m_1(\frac{1}{2}-2s)}\||D_x|^{\delta}g_j^\vee\|
_{L^\frac{1}{\delta}_T(L^\frac{2}{1-2\delta}_{x,y})}
\||D_x|^{\frac{1}{2}-\delta}\phi_{1,j_1,m_1}\|_{L^\frac{2}{1-2\delta}_T(L^\frac{1}{\delta}_{x,y})}\\
&\quad\quad\quad\quad\times\|\phi_{2,j_2}\|_{L^2}\\
&\lesssim \sum_{j_1,j_2\geq0}\sum_{j_2\geq
\max(j,2m_1)\geq0}2^{j(-\frac{1}{2}+)}2^{m_1(\frac{1}{2}-2s)}2^{j/2}2^{j_1/2}
\|g_j\|_{L^2}
\|\phi_{1,j_1,m_1}\|_{L^2}\|\phi_{2,j_2}\|_{L^2}\\
&\lesssim\|\phi_1\|_{X_{0,\frac{1}{2}}}\|\phi_2\|_{X_{0,\frac{1}{2}}}.
\end{split}
\end{equation*}

\noindent If $j\leq j_2$ and $\frac{1}{4}<s\leq1$, we bound
(\ref{2}) by
\begin{equation*}
\begin{split}
&\sum_{j_1,j_2\geq0}\sum_{j_2\geq
\max(j,2m_1)\geq0}2^{j(-\frac{1}{2}+)}2^{m_1(\frac{1}{2}-2s)}\|g_j\|
_{L^2} \||D_x|^{\frac{1}{4}}\phi_{1,j_1,m_1}\|_{L^4_T(L^4_{x,y})}
\||D_x|^{\frac{1}{4}}\phi_{2,j_2}\|_{L^4_T(L^4_{x,y})}\\
&\lesssim \sum_{j_1,j_2\geq0}\sum_{j_2\geq
\max(j,2m_1)\geq0}2^{j(-\frac{1}{2}+)}2^{m_1(\frac{1}{2}-2s)}2^{j_2/2}2^{j_1/2}
\|g_j\|_{L^2}
\|\phi_{1,j_1,m_1}\|_{L^2}\|\phi_{2,j_2}\|_{L^2}\\
&\lesssim\|\phi_1\|_{X_{0,\frac{1}{2}}}\|\phi_2\|_{X_{0,\frac{1}{2}}}.
\end{split}
\end{equation*}

\noindent If $j>j_2$, we also have
\begin{equation*}
\begin{split}
(\ref{2})&\lesssim
\sum_{j,j_1\geq0}\sum_{j>\max(j_2,2m_1)\geq0}2^{j(-\frac{1}{2}+)}2^{m_1(\frac{1}{2}-2s)}
\||D_x|^{\frac{1}{4}}\phi_{2,j_2}\|_{L^4_T(L^4_{x,y})}
\||D_x|^\frac{1}{4}\phi_{1,j_1,m_1}\|_{L^4_T(L^4_{x,y})}\\
&\quad\quad\quad\quad\times\|g_j\|_{L^2}\\
&\lesssim\sum_{j,j_1\geq0}\sum_{j>\max(j_2,2m_1)\geq0}2^{j(-\frac{1}{2}+)}
2^{j\max(0,\frac{1}{4}-s)}2^{j_1/2}2^{j_2/2}\|g_j\|_{L^2}
\|\phi_{1,j_1,m_1}\|_{L^2}\|\phi_{2,j_2}\|_{L^2}\\
&\lesssim\|\phi_1\|_{X_{0,\frac{1}{2}}}\|\phi_2\|_{X_{0,\frac{1}{2}}}.
\end{split}
\end{equation*}

\vspace{0.3cm}

\noindent{\bf Subsubsubcase 2:} $\max(j,j_2)<2m_1$.

In this case, the argument in case A can still work by replacing the
$\frac{1}{4}$ derivative on $g_j$ by $\frac{1}{4}$ derivative on
$\phi_1$ when $\frac{1}{4}<s\leq1$, and
$\Big|\frac{\mu_1}{\xi_1}-\frac{\mu_2}{\xi_2}\Big|^2>
\frac{1}{2}|\xi_1+\xi_2|^2[5(\xi_1^2+\xi_1\xi_2+\xi_2^2)-3\alpha].$
We omit the rest details.

\vspace{0.3cm}

\noindent {\bf Subcase B2:} $|\mu_1|\geq|\mu_2|$. One can use the
same argument presented in Case B1  by inverting the role of
$(\xi_1,\mu_1)$ and $(\xi_2,\mu_2)$.

\vspace{5mm}

\noindent {\bf High-Low interaction}

\vspace{0.3cm}

\noindent In this domain, the estimates will be more complicated.
Roughly speaking, we will consider the term
$\frac{|\mu_2|}{|\xi_2|}$ in two regions,
$\frac{|\mu_2|}{|\xi_2|}\gtrsim\max(|\xi_1|^2,\frac{|\mu_1|}{|\xi_1|})$
and
$\frac{|\mu_2|}{|\xi_2|}\ll\max(|\xi_1|^2,\frac{|\mu_1|}{|\xi_1|}).$

\vspace{0.3cm}

\noindent Region I:
$\frac{|\mu_2|}{|\xi_2|}\gtrsim\max\Big(|\xi_1|^2,\frac{|\mu_1|}{|\xi_1|}\Big)$.

\vspace{0.3cm}

\noindent {\bf Case A:} $|\xi_1+\xi_2|^2\gtrsim
\frac{|\mu_1+\mu_2|}{|\xi_1+\xi_2|}.$

\vspace{0.3cm}

\noindent {\bf Subcase A1:}
$|\xi_1|^2\gtrsim\frac{|\mu_1|}{|\xi_1|}.$

\vspace{0.3cm}

We apply the dyadic decomposition with respect to
$|\xi|\sim|\xi_1|\sim 2^{m_1}$ to bound (\ref{1}) by
\begin{equation}\label{eq:3.13}
\begin{split}
\sum_{j_1,j_2,j\geq0}\sum_{m_1\geq0}&2^{j(-\frac{1}{2}+)}\int_{A_3}g_j(\xi,\mu,\tau)
\chi_j(\tau-\omega(\xi,\mu))\chi_1(\xi,\mu)\\
&\times 2^{m_1}\hat\phi_{1,j_1,m_1}(\xi_1,\mu_1,\tau_1)
\frac{\hat\phi_{2,j_2}(\xi_2,\mu_2,\tau_2)}{\Big(\frac{|\mu_2|}{|\xi_2|}\Big)^s}.
\end{split}
\end{equation}

\vspace{0.3cm}

\noindent {\bf Subsubcase A1a:} $|\xi_2|\geq 1$ and $\max(j,j_2)\geq
\frac{3}{2}m_1$.

\vspace{0.3cm}

\noindent We first notice that
\begin{equation*}
\begin{split}
(\ref{1})&\lesssim\sum_{j_1,m_1\geq0}\sum_{\max(j,j_2)\geq\frac{3}{2}m_1\geq0}2^{j(-\frac{1}{2}+)}
\int_{A_3}g_j(\xi,\mu,\tau)\chi_j(\tau-\omega(\xi,\mu))\chi_1(\xi,\mu)\\
&\quad\quad\quad\quad\times
2^{m_1(1-2s)}\hat\phi_{1,j_1,m_1}(\xi_1,\mu_1,\tau_1)
\hat\phi_{2,j_2}(\xi_2,\mu_2,\tau_2).
\end{split}
\end{equation*}
If $j\geq j_2$,
\begin{equation*}
\begin{split}
(\ref{1})&\lesssim\sum_{j_1,j\geq0}\sum_{j\geq
\max(j_2,\frac{3}{2}m_1)\geq0}2^{j(-\frac{1}{2}+)}2^{m_1(\frac{3}{4}-2s)}
\|g_j\|_{L^2}\||D_x|^\frac{1}{4}\phi_{1,j_1,m_1}\|_{L^4_T(L^4_{x,y})}
\||D_x|^\frac{1}{4}\phi_{2,j_2}\|_{L^4_T(L^4_{x,y})}\\
&\lesssim\sum_{j_1,j\geq0}\sum_{j\geq
\max(j_2,\frac{3}{2}m_1)\geq0}2^{j(-\frac{1}{2}+)}2^{m_1(\frac{3}{4}-2s)}
\|g_j\|_{L^2}2^{j_1/2}2^{j_2/2}\|\phi_{1,j_1,m_1}\|_{L^2}\|\phi_{2,j_2}\|_{L^2}.\\
&\lesssim
\|\phi_1\|_{X_{0,\frac{1}{2}}}\|\phi_2\|_{X_{0,\frac{1}{2}}}.
\end{split}
\end{equation*}
If $j<j_2$,
\begin{equation*}
\begin{split}
(\ref{1})&\lesssim\sum_{j_1,j_2\geq0}\sum_{j_2\geq\max(\frac{3}{2}m_1,j)\geq0}2^{j(-\frac{1}{2}+)}2^{m_1(\frac{1}{2}-2s)}
\||D_x|^\frac{1}{4}g_j\|_{L^4_T(L^4_{x,y})}\||D_x|^\frac{1}{4}\phi_{1,j_1,m_1}\|_{L^4_T(L^4_{x,y})}\\
&\quad\quad\times\|\phi_{2,j_2}\|_{L^2}\\
&\lesssim\sum_{j_1,j_2\geq0}\sum_{j_2\geq
\max(j,\frac{3}{2}m_1)\geq0}2^{j(-\frac{1}{2}+)}2^{j/2}2^{m_1(\frac{1}{2}-2s)}2^{j_1/2}
\|g_j\|_{L^2}\|\phi_{1,j_1,m_1}\|_{L^2}\|\phi_{2,j_2}\|_{L^2}.\\
&\lesssim
\|\phi_1\|_{X_{0,\frac{1}{2}}}\|\phi_2\|_{X_{0,\frac{1}{2}+}}.
\end{split}
\end{equation*}

\vspace{0.3cm}

\noindent {\bf Subsubcase A1b:} $|\xi_2|\geq 1$ and $\max(j,j_2)\leq
\frac{3}{2}m_1.$

\vspace{0.3cm}

\noindent As in the estimates in the high frequency interaction
domain, it is necessary to consider more cases.

\vspace{0.3cm}

\noindent {\bf Subsubsubcase 1:}
$\Big|\frac{\mu_1}{\xi_1}-\frac{\mu_2}{\xi_2}\Big|^2<\frac{1}{2}|\xi_1+\xi_2|^2
|5(\xi_1^2+\xi_1\xi_2+\xi_2^2)-3\alpha|.$

\vspace{0.3cm}

\noindent In this case, the resonant interaction does not happen. By
inequality (\ref{5KP-Iresonance}) and $|\xi_2|>1$, we get that
$j_1=\max(j,j_1,j_2)\geq 4m_1.$ We now bound (\ref{1}) by
\begin{equation*}
\begin{split}
&\sum_{j_1\geq0}\sum_{j_1\geq4m_1\geq0}\sum_{\frac{3}{2}m_1\geq\max(j,j_2)\geq0}2^{j(-\frac{1}{2}+)}
2^{m_1(\frac{3}{4}-2s)}
\||D_x|^\frac{1}{4}g_j^\vee\|_{L^4_T(L^4_{x,y})}\|\phi_{1,j_1,m_1}\|_{L^2}\\
&\quad\quad\times\||D_x|^\frac{1}{4}\phi_{2,j_2}\|_{L^4_T(L^4_{x,y})}\\
&\lesssim\sum_{j_1\geq0}\sum_{j_1\geq4m_1\geq0}\sum_{\frac{3}{2}m_1\geq\max(j,j_2)\geq0}
2^{j(-\frac{1}{2}+)}2^{j/2}2^{m_1(\frac{3}{4}-2s)}2^{j_2/2}
\|g_j\|_{L^2}\|\phi_{1,j_1,m_1}\|_{L^2}\|\phi_{2,j_2}\|_{L^2}.\\
&\lesssim
\|\phi_1\|_{X_{0,\frac{1}{2}+}}\|\phi_2\|_{X_{0,\frac{1}{2}}}.
\end{split}
\end{equation*}

\vspace{0.3cm}

\noindent {\bf Subsubsubcase 2:}
$\Big|\frac{\mu_1}{\xi_1}-\frac{\mu_2}{\xi_2}\Big|^2\geq\frac{1}{2}|\xi_1+\xi_2|^2
|5(\xi_1^2+\xi_1\xi_2+\xi_2^2)-3\alpha|$.

\vspace{0.3cm}

\noindent We need to divide the estimate into two cases:
$$\Big|5(\xi_1^4-\xi_2^4)-3\alpha(\xi_1^2-\xi_2^2)-
\Big[\Big(\frac{\mu_1}{\xi_1}\Big)^2-\Big(\frac{\mu_2}{\xi_2}\Big)^2\Big]\Big|\geq1$$
and $$\Big|5(\xi_1^4-\xi_2^4)-3\alpha(\xi_1^2-\xi_2^2)-
\Big[\Big(\frac{\mu_1}{\xi_1}\Big)^2-\Big(\frac{\mu_2}{\xi_2}\Big)^2\Big]\Big|<1.$$
As we known, the first inequality means the determinant of the
Jacobian of the change of variables (\ref{eq:3.5}) $|J_\mu|\geq1$.
So we get
\begin{equation*}
\begin{split}
(\ref{1})&\lesssim
\sum_{m_1,j_1\geq0}\sum_{\frac{3}{2}m_1>\max(j_2,j)\geq0}2^{j(-\frac{1}{2}+)}2^{j/2-m_1}2^{m_1(1-2s)}2^{j_1/2}2^{j_2/2}\|g_j\|_{L^2}
\|\phi_{1,j_1,m_1}\|_{L^2}\|\phi_{2,j_2}\|_{L^2}.\\
&\lesssim\|\phi_1\|_{X_{0,\frac{1}{2}}}\|\phi_2\|_{X_{0,\frac{1}{2}}}.
\end{split}
\end{equation*}
For the second inequality, we recur to the change of variables
(\ref{eq:3.7}). In the same way, we get
\begin{equation*}
\begin{split}
(\ref{1})&\lesssim
\sum_{m_1,j_1\geq0}\sum_{2m_1>\max(j_2,j)\geq0}2^{j(-\frac{1}{2}+)}2^{-2m_1}2^{m_1(1-2s)}2^{j_1/2}2^{j_2/2}\|g_j\|_{L^2}
\|\phi_{1,j_1,m_1}\|_{L^2}\|\phi_{2,j_2}\|_{L^2}.\\
&\lesssim\|\phi_1\|_{X_{0,\frac{1}{2}}}\|\phi_2\|_{X_{0,\frac{1}{2}}}.
\end{split}
\end{equation*}

\vspace{0.3cm}

\noindent {\bf Subsubcase A1c:} $|\xi_2|<1$.

\vspace{0.3cm}

\noindent If $|\mu_2|\lesssim1$, since
$\frac{|\mu_2|}{|\xi_2|}\gtrsim |\xi_1|^2$, we have that
$|\xi_2|\lesssim |\xi_1|^{-2}$. Thus we bound (\ref{1}) by
\begin{equation*}
\begin{split}
&\sum_{j,j_1,j_2\geq 0}\sum_{m_1\geq
0}2^{j(-\frac{1}{2}+)}2^{m_1(1-2s)}\|g_j\|_{L^2}\|\phi_{1,j_1,m_1}\|_{L^\infty_T(L^2_{x,y})}
\|(m(\xi_2,\mu_2)\hat\phi_{2,j_2})^\vee\|_{L^2_T(L^\infty_{x,y})}\\
\lesssim&\sum_{j,j_1,j_2\geq 0}\sum_{m_1\geq
0}2^{j(-\frac{1}{2}+)}2^{m_1(1-2s)}\|g_j\|_{L^2}2^{j_1/2}\|\phi_{1,j_1,m_1}\|_{L^2}2^{-m_1}\|\phi_{2,j_2}\|_{L^2}\\
\lesssim&\|\phi_1\|_{X_{0,\frac{1}{2}}}\|\phi_2\|_{X_{0,\frac{1}{2}}}.
\end{split}
\end{equation*}
Here $m(\xi_2,\mu_2)$ denotes the characteristic function of set
$\{|\xi_2|\lesssim 2^{-2m_1},|\mu_2|<1\}$.

\noindent If $|\mu_2|\gtrsim 1$ and $\max(j,j_2)\geq m_1$, when
$j=\max(j,j_2)$, we choose $\min(\frac{1}{2},s)>\delta>0$  such that
$\frac{1}{2}-2s+\delta<|-\frac{1}{2}+|$ and bound (\ref{1}) by
\begin{equation*}
\begin{split}
&\sum_{j_1,j\geq0}\sum_{0\leq\max(j_2, m_1)\leq
j}2^{j(-\frac{1}{2}+)}2^{m_1(1/2-2s+\delta)}
\|g_j\|_{L^2}\||D_x|^{\frac{1}{2}-\delta}\phi_{1,j_1,m_1}\|_{L^\frac{2}{1-2\delta}_T(L^{\frac{1}{\delta}}_{x,y})}\\
&\quad\quad\times\||D_x|^{\delta}\phi_{2,j_2}\|_{L^{\frac{1}{\delta}}_T(L^\frac{2}{1-2\delta}_{x,y})}\\
&\lesssim \sum_{j_1,j\geq0}\sum_{0\leq \max(j_2, m_1)\leq
j}2^{j(-\frac{1}{2}+)}2^{m_1(1/2-2s+\delta)}2^{j_1/2}2^{j_2/2}\|g_j\|_{L^2}
\|\phi_{1,j_1,m_1}\|_{L^2}\|\phi_{2,j_2}\|_{L^2}\\
&\lesssim\|\phi_1\|_{X_{0,\frac{1}{2}}}\|\phi_2\|_{X_{0,\frac{1}{2}}}.
\end{split}
\end{equation*}
While for the case $j_2=\max(j,j_2)$, we bound (\ref{1}) by
\begin{equation*}
\begin{split}
&\sum_{j_1,j_2\geq0}\sum_{0\leq \max(j, m_1)\leq
j_2}2^{j(-\frac{1}{2}+)}2^{m_1(1/2-2s)}
\||D_x|^{\delta}g_j^\vee\|_{L^{\frac{1}{\delta}}_T(L^\frac{2}{1-2\delta}_{x,y})}\\
&\quad\quad\times\||D_x|^{\frac{1}{2}-\delta}\phi_{1,j_1,m_1}\|_{L^\frac{2}{1-2\delta}_T(L^{\frac{1}{\delta}}_{x,y})}
\|\phi_{2,j_2}\|_{L^2}\\
&\lesssim \sum_{j_1,j_2\geq0}\sum_{0\leq \max(j, m_1)\leq
j_2}2^{j(-\frac{1}{2}+)}2^{m_1(1/2-2s)}2^{j/2}2^{j_1/2}\|g_j\|_{L^2}
\|\phi_{1,j_1,m_1}\|_{L^2}\|\phi_{2,j_2}\|_{L^2}\\
&\lesssim\|\phi_1\|_{X_{0,\frac{1}{2}}}\|\phi_2\|_{X_{0,\frac{1}{2}+}}.
\end{split}
\end{equation*}
If $|\mu_2|\gtrsim 1$ and $\max(j,j_2)<m_1$, we have to divided two
subcases to estimate (\ref{1}).

\vspace{0.3cm}

\noindent {\bf Subsubsubcase a:}
$\Big|\frac{\mu_1}{\xi_1}-\frac{\mu_2}{\xi_2}\Big|^2<\frac{1}{2}
|\xi_1+\xi_2|^2|5[\xi_1^2+\xi_1\xi_2+\xi_2^2]-3\alpha|$.

\vspace{0.3cm}

\noindent As we know, the estimate on the resonance function can be
used now. We have $|\xi_1|^4|\xi_2|\lesssim 2^{\max(j,j_1,j_2)}$.
Unfortunately, since $|\xi_2|<1$, the element inequality is not as
good as we have used. We claim that $|\xi_2|\geq |\xi_1|^{-2}$.
Otherwise, if $|\frac{\mu_1}{\xi_1}|\sim |\frac{\mu_2}{\xi_2}|$,
then $|\mu_2|\lesssim|\xi_1|^2|\xi_2|\lesssim 1$. And if
$|\frac{\mu_1}{\xi_1}|\ll |\frac{\mu_2}{\xi_2}|$, since we are in
subcase a:
$\Big|\frac{\mu_1}{\xi_1}-\frac{\mu_2}{\xi_2}\Big|^2<\frac{1}{2}
|\xi_1+\xi_2|^2|5[\xi_1^2+\xi_1\xi_2+\xi_2^2]-3\alpha|$, we have
$|\mu_2|\lesssim |\xi_1|^2|\xi_2|\lesssim 1$. These conflict with
the assumption $|\mu_2|\gtrsim 1$. Thus we have $2^{2m_1}\leq
2^{\max(j,j_1,j_2)}$. It is clear that $j_1=\max(j,j_1,j_2)$. We
bound (\ref{1}) with
\begin{equation*}
\begin{split}
&\sum_{j_1\geq j,j_2\geq0}\sum_{0\leq 2m_1\leq
j_1}2^{j(-\frac{1}{2}+)}2^{m_1(1/2-2s+\delta)}
\||D_x|^{\frac{1}{2}-\delta}g_j^\vee\|_{L^\frac{2}{1-2\delta}_T(L^{\frac{1}{\delta}}_{x,y})}
\|\phi_{1,j_1,m_1}\|_{L^2}\\
&\quad\quad\times\||D_x|^{\delta}\phi_{2,j_2}\|_{L^{\frac{1}{\delta}}_T(L^\frac{2}{1-2\delta}_{x,y})}\\
&\lesssim \sum_{j_1\geq j,j_2\geq0}\sum_{0\leq 2m_1\leq
j_1}2^{j(-\frac{1}{2}+)}2^{m_1(1/2-2s+\delta)}2^{j/2}2^{j_2/2}\|g_j\|_{L^2}
\|\phi_{1,j_1,m_1}\|_{L^2}\|\phi_{2,j_2}\|_{L^2}\\
&\lesssim\|\phi_1\|_{X_{0,\frac{1}{2}+}}\|\phi_2\|_{X_{0,\frac{1}{2}}}.
\end{split}
\end{equation*}

\vspace{0.3cm}

\noindent {\bf Subsubsubcase b:}
$\Big|\frac{\mu_1}{\xi_1}-\frac{\mu_2}{\xi_2}\Big|^2\geq\frac{1}{2}
|\xi_1+\xi_2|^2|5[\xi_1^2+\xi_1\xi_2+\xi_2^2]-3\alpha|$. In this
case, one can run the same argument in the case A1b.

\vspace{0.3cm}

\noindent {\bf Subsubcase A2:}
$|\xi_1|^2\ll\frac{|\mu_1|}{|\xi_1|}$.

\vspace{0.3cm}

The argument in Subsubcase A1 above can also help us to get the same
estimates. We would like to show the different point when we
encounter the case $|\mu_2|\gtrsim1$,
$\Big|\frac{\mu_1}{\xi_1}-\frac{\mu_2}{\xi_2}\Big|^2<\frac{1}{2}|\xi_1+\xi_2|^2
|5[\xi_1^2+\xi_1\xi_2+\xi_2^2]-3\alpha|$. Here we still have
$|\xi_1|^4|\xi_2|\lesssim 2^{\max(j,j_1,j_2)}$. If
$|\mu_1|\lesssim|\mu_2|$, we have
$\frac{|\mu_2|}{|\xi_2|}\gg\frac{|\mu_1|}{|\xi_1|}$. It means that
we also have $|\xi_2|>|\xi_1|^{-2}.$ If $|\mu_1|\gg|\mu_2|$, then we
have $|\mu_1+\mu_2|\sim|\mu_1|$, thus
$\frac{|\mu_1+\mu_2|}{|\xi_1+\xi_2|}\gg|\xi_1+\xi_2|^2$. This does
not appear since we are in case
$|\xi_1+\xi_2|^2\gtrsim\frac{|\mu_1+\mu_2|}{|\xi_1+\xi_2|}.$ Then we
can run the argument in Subsubcase A1.

\vspace{0.3cm}

\noindent {\bf Case B:} $|\xi_1+\xi_2|^2\ll
\frac{|\mu_1+\mu_2|}{|\xi_1+\xi_2|}.$

\vspace{0.3cm}

\noindent {\bf Subcase B1:} $|\mu_1|\lesssim|\mu_2|$.

\vspace{0.3cm}

In this region, we also have $\frac{|\mu_2|}{|\xi_2|}\gg |\xi_1|^2$.
Similar to the argument presented in the second part of Case B1 of
domain $A_2$, we can bound (\ref{2}) with
\begin{equation*}
\begin{split}
\sum_{j_1,j_2,j\geq0}\sum_{m_1\geq0}&2^{j(-\frac{1}{2}+)}
\int_{A_3}g_j(\xi,\mu,\tau)\chi_j(\tau-\omega(\xi,\mu))\chi_2(\xi,\mu)\\
&\times |\xi_1+\xi_2|^{1-s}2^{-m_1s}
\hat\phi_{1,j_1,m_1}(\xi_1,\mu_1,\tau_1)
\hat\phi_{2,j_2}(\xi_2,\mu_2,\tau_2).
\end{split}
\end{equation*} Then the estimate in case A above works.

\vspace{0.3cm}

\noindent {\bf Subcase B2:} $|\mu_1|\gg|\mu_2|$.

\vspace{0.3cm}

It is clear that
$\frac{|\mu_1|}{|\xi_1|}\sim\frac{|\mu_1+\mu_2|}{|\xi_1+\xi_2|}.$
Thus (\ref{2}) by
\begin{equation*}
\begin{split}
\sum_{j_1,j_2,j\geq0}\sum_{m_1\geq0}&2^{j(-\frac{1}{2}+)}\int_{A_3}
g_j(\xi,\mu,\tau)\chi_j(\tau-\omega(\xi,\mu))\chi_2(\xi,\mu)\\
&\times
2^{m_1(1-2s)}\hat\phi_{1,j_1,m_1}(\xi_1,\mu_1,\tau_1)\hat\phi_{2,j_2}(\xi_2,\mu_2,\tau_2).
\end{split}
\end{equation*}
If $|\mu_2|<1$, then we also have $|\xi_2|\leq |\xi_1|^{-2}$. By the
same argument in case A1c, we bound (\ref{2}) by $
\|\phi_1\|_{X_{0,\frac{1}{2}}}\|\phi_2\|_{X_{0,\frac{1}{2}+}}.$

\noindent If $|\mu_2|\geq 1$, the estimates in case A above can also
work until we come to the case $|\xi_2|<|\xi_1|^{-2},$
$\max(j,j_2)<2m_1$ and
$\Big|\frac{\mu_1}{\xi_1}-\frac{\mu_2}{\xi_2}\Big|^2<\frac{1}{2}
|\xi_1+\xi_2|^2|5[\xi_1^2+\xi_1\xi_2+\xi_2^2]-3\alpha|$. Of course,
in this case, the estimate on resonance function can also bring us
$$|\xi_1|^4|\xi_2|\lesssim 2^{\max(j,j_1,j_2)}.$$
But this estimate can not help us to get any benefit since
$|\xi_2|<|\xi_1|^{-2}.$ Fortunately, in this case, for fixed
$\mu_1,\xi_1,\xi_2,$ the variable $\mu_2$ can range in two symmetry
intervals with length $\Delta_{\mu_2}\lesssim
|\xi_1|^2|\xi_2|\lesssim1.$ Represent the change of variables
(\ref{eq:3.14}) here,
\begin{equation*}
\begin{split}
&\sum_{j_1,m_1\geq0}\sum_{2m_1>\max(j,j_2)\geq0}2^{j(-\frac{1}{2}+)}2^{m_1(1-2s)}\int
g_j(\xi,\mu,\theta_1+\omega(\xi_1,\mu_1)+\theta_2+\omega(\xi_2+\mu_2))
\\
&\times\chi_j(\theta_1+\theta_2+\omega(\xi_1,\mu_2)+\omega(\xi_2+\mu_2)-
\omega(\xi_1+\xi_2,\mu_1+\mu_2))
\\
&\times\hat\phi_{1,j_1,m_1}(\xi_1,\mu_1,\theta_1+\omega(\xi_1,\mu_1))
\hat\phi_{2,j_2}(\xi_2,\mu_2,\theta_2+\omega(\xi_2,\tau_2))d\xi_1d\mu_1d\xi_2d\mu_2d\theta_1d\theta_2.
\end{split}
\end{equation*}
By Cauchy-Schwarz inequality, we control the integral
(\ref{eq:3.14}) by
\begin{equation*}
\begin{split}
&\|\phi_{1,j_1,m_1}\|_{L^2}\left(\int\Big| \int
H(\xi_1,\xi_2,\mu_1,\mu_2,\theta_1,\theta_2)
d\xi_2d\mu_2d\theta_2\Big|^2
d\xi_1d\mu_1d\theta_1\right)^\frac{1}{2}\\
&\lesssim2^{j_2/2}2^{-m_1}\|\phi_{1,j_1,m_1}\|_{L^2}\left(\int\int
|H(\xi_1,\xi_2,\mu_1,\mu_2,\theta_1,\theta_2)|^2
d\xi_2d\mu_2d\theta_2 d\xi_1d\mu_1d\theta_1\right)^\frac{1}{2}\\
&\lesssim2^{j_2/2}2^{-m_1}\|\phi_{1,j_1,m_1}\|_{L^2}\|g_j\|_{L^2}\|\phi_{2,j_2}\|.
\end{split}
\end{equation*}
Here $H(\xi_1,\xi_2,\mu_1,\mu_2,\theta_1,\theta_2)$ denotes
$g_j(\xi,\mu,\theta_1+\omega(\xi_1,\mu_1)+\theta_2+\omega(\xi_2+\mu_2))
\chi_j(\theta_1+\theta_2+\omega(\xi_1,\mu_2)+\omega(\xi_2+\mu_2)-
\omega(\xi_1+\xi_2,\mu_1+\mu_2))\hat\phi_{2,j_2}(\xi_2,\mu_2,\theta_2+\omega(\xi_2,\tau_2)).$
Now we put this estimate into the summation above to obtain
$$(\ref{2})\lesssim \|\phi_1\|_{X_{0,\frac{1}{2}+}}\|\phi_2\|_{X_{0,\frac{1}{2}}}.$$

\vspace{0.5cm}

\noindent Region II:
$\frac{|\mu_2|}{|\xi_2|}\ll\max\Big(|\xi_1|^2,\frac{|\mu_1|}{|\xi_1|}\Big).$

\vspace{0.3cm}

\noindent {\bf Case A:} $|\xi_1|^2\gg\frac{|\mu_1|}{|\xi_1|}$.

\vspace{0.3cm}

\noindent Since $|\xi_1+\xi_2|^3\sim |\xi_1|^3\gg |\mu_1|$ and
$|\xi_1|^3\gg|\xi_1|^2|\xi_2|\gg|\mu_2|$, the resonant interaction
does not happen, so $2^{\max(j,j_1,j_2)}\geq |\xi_1|^4|\xi_2|.$

\noindent If $|\mu_2|<1$ and $j=\max(j,j_1,j_2)$, then $2^j\geq
2^{4m_1+m_2}$. In the same way, we bound (\ref{1}) by
\begin{equation*}
\begin{split}
&\sum_{j,j_1,j_2\geq 0}\sum_{j\geq
4m_1+m_2}\sum_{m_2<m_1}2^{j(-\frac{1}{2}+)}2^{m_1}\|g_j\|_{L^2}\|\phi_{1,j_1,m_1}\|_{L^\infty_T(L^2_{x,y})}\\
&\quad\quad\quad\times\|(m_{m_2}(\xi_2,\mu_2)\hat\phi_{2,j_2})^\vee\|_{L^2_T(L^\infty_{x,y})}\\
\lesssim&\sum_{j,j_1,j_2\geq 0}\sum_{j\geq4m_1+m_2}
\sum_{m_2<m_1}2^{j(-\frac{1}{2}+)}2^{m_1}\|g_j\|_{L^2}2^{j_1/2}
\|\phi_{1,j_1,m_1}\|_{L^2}2^{m_2/2}\|\phi_{2,j_2}\|_{L^2}\\
\lesssim&\|\phi_1\|_{X_{0,\frac{1}{2}}}\|\phi_2\|_{X_{0,\frac{1}{2}+}}.
\end{split}
\end{equation*}
Here we used Proposition \ref{prop5} with $m_{m_2}$ denoting a class
of multipliers which are the characteristic functions of the sets
$\{|\xi_2|\sim 2^{m_2},|\mu_2|<1\}$.

\noindent If $|\mu_2|<1$ and $j_1=\max(j,j_1,j_2)$ or
$j_2=\max(j,j_1,j_2)$ is the maximal value, similarly we have
\begin{equation*}
\begin{split}
(\ref{1})&\lesssim\sum_{j,j_1,j_2\geq 0}\sum_{j_1\geq
4m_1+m_2}\sum_{m_2<m_1}2^{j(-\frac{1}{2}+)}2^{m_1}\|\phi_{1,j_1,m_1}\|_{L^2}\|g_j^\vee\|_{L^\infty_T(L^2_{x,y})}\\
&\quad\quad\times\|(m_{m_2}(\xi_2,\mu_2)\hat\phi_{2,j_2})^\vee\|_{L^2_T(L^\infty_{x,y})}\\
\lesssim&\sum_{j,j_1,j_2\geq 0}\sum_{j_1\geq
4m_1+m_2}\sum_{m_2<m_1}2^{j(-\frac{1}{2}+)}2^{m_1}2^{j/2}2^{m_2/2}\|g_j\|_{L^2}\|\phi_{1,j_1,m_1}\|_{L^2}
\|\phi_{2,j_2}\|_{L^2}\\
\lesssim&\|\phi_1\|_{X_{0,\frac{1}{2}}}\|\phi_2\|_{X_{0,\frac{1}{2}+}}.
\end{split}
\end{equation*}

\noindent If $|\mu_2|\geq 1$ and $\max(j,j_2)\geq 2m_1$, let
$j=\max(j,j_2)$, there exists $\min(\frac{1}{2},s)>\delta>0$ and
$|-\frac{1}{2}+|>\frac{1}{4}+\frac{1}{2}\delta>0$ such that
\begin{equation*}
\begin{split}
(\ref{1})&\lesssim\sum_{j_1,j\geq0}\sum_{0\leq\max(j_2, 2m_1)\leq
j}2^{j(-\frac{1}{2}+)}2^{m_1(1/2+\delta)}
\|g_j\|_{L^2}\||D_x|^{\frac{1}{2}-\delta}\phi_{1,j_1,m_1}\|_{L^\frac{2}{1-2\delta}_T(L^{\frac{1}{\delta}}_{x,y})}\\
&\quad\quad\times\||D_x|^{\delta}\phi_{2,j_2}\|_{L^{\frac{1}{\delta}}_T(L^\frac{2}{1-2\delta}_{x,y})}\\
&\lesssim \sum_{j_1,j\geq0}\sum_{0\leq \max(j_2, 2m_1)\leq
j}2^{j(-\frac{1}{2}+)}2^{m_1(1/2+\delta)}2^{j_1/2}2^{j_2/2}\|g_j\|_{L^2}
\|\phi_{1,j_1,m_1}\|_{L^2}\|\phi_{2,j_2}\|_{L^2}\\
&\lesssim\|\phi_1\|_{X_{0,\frac{1}{2}}}\|\phi_2\|_{X_{0,\frac{1}{2}}}.
\end{split}
\end{equation*}
For the case $j_2=\max(j,j_2)$, we bound (\ref{1}) by
\begin{equation*}
\begin{split}
&\sum_{j_1,j_2\geq0}\sum_{0\leq\max(j, 2m_1)\leq
j_2}2^{j(-\frac{1}{2}+)}2^{m_1/2}
\||D_x|^\frac{1}{4}g_j^\vee\|_{L^4_T(L^4_{x,y})}\||D_x|^\frac{1}{4}\phi_{1,j_1,m_1}\|_{L^4_T(L^4_{x,y})}
\|\phi_{2,j_2}\|_{L^2}\\
&\lesssim \sum_{j_1,j_2\geq0}\sum_{0\leq\max(j, 2m_1)\leq
j_2}2^{j(-\frac{1}{2}+)}2^{m_1/2}2^{j/2}2^{j_1/2}\|g_j\|_{L^2}
\|\phi_{1,j_1,m_1}\|_{L^2}\|\phi_{2,j_2}\|_{L^2}\\
&\lesssim\|\phi_1\|_{X_{0,\frac{1}{2}}}\|\phi_2\|_{X_{0,\frac{1}{2}}}.
\end{split}
\end{equation*}

\noindent If $|\mu_2|\geq1$ and $\max(j,j_2)< 2m_1$, we would like
to perform a dyadic decomposition by setting $|\xi_i|\sim 2^{m_i}$
with $i=1,2$ and $m_1\geq0, m_2\in\mathbb{Z}$. The dyadic
decomposition with respect to $|\mu_2|\sim 2^{n_2}, n_2\geq0$ will
be useful. Another useful note is that $m_2^*=\max(n_2-m_2, 2m_2).$

We perform the change of variables (\ref{eq:3.5}). It is easy to see
that $|J_\mu|> |\xi_1|^4$, so
\begin{equation*}
\begin{split}
(\ref{1})&\lesssim\sum_{j_1,m_1,n_2\geq
0}\sum_{m_2}\sum_{2m_1>\max(j,j_2)\geq0}2^{j(-\frac{1}{2}+)}2^{m_1}\int
g_j(u,v,w)\chi_j(u,v,w)\\
&\times|J_\mu|^{-1}H(u,v,w,\mu_2,\theta_1,\theta_2)dudvdwd\mu_2d\theta_1d\theta_2\\
&\lesssim \sum_{j_1,m_1,n_2\geq
0}\sum_{m_2}\sum_{2m_1>\max(j,j_2)\geq0}2^{j(-\frac{1}{2}+)}2^{-m_1}2^{n_2/2}2^{j_1/2}2^{j_2/2}\\
&\times \|g_j\|_{L^2}\|\phi_{1,j_1,m_1}\|_{L^2}
\frac{\|\phi_{2,j_2,m_2,n_2}\|_{L^2}}{\max(1,2^{m_2^*})^s}.
\end{split}
\end{equation*}
If $m_2\geq 0$ and $n_2-m_2<0$, we bound (\ref{1}) with
\begin{equation*}
\begin{split}
&\sum_{j,j_1,j_2,m_1\geq 0}\sum_{0\leq n_2\leq
m_2<m_1}2^{j(-\frac{1}{2}+)}
2^{-m_1}2^{n_2/2}2^{(-2m_2)s}\\
&\times
2^{j_1/2}2^{j_2/2}\|g_j\|_{L^2}\|\phi_{1,j_1,m_1}\|_{L^2}\|\phi_{2,j_2,m_2,n_2}\|_{L^2}\\
&\lesssim\|\phi_1\|_{X_{0,\frac{1}{2}+}}\|\phi_2\|_{X_{0,\frac{1}{2}}}.
\end{split}
\end{equation*}
 If $2m_2\geq n_2-m_2\geq
0$ and $j>2m_2$, we have
\begin{equation*}
\begin{split}
(\ref{1})&\lesssim\sum_{j_1,j_2\geq 0}\sum_{0\leq n_2-
m_2<2m_2}\sum_{j\geq 2m_2\geq0}\sum_{m_1\geq m_2\geq 0}
2^{j(-\frac{1}{2}+)}2^{-m_1}2^{(n_2-m_2)/2}2^{m_2/2}2^{-2m_2s}\\
&\times
2^{j_1/2}2^{j_2/2}\|g_j\|_{L^2}\|\phi_{1,j_1,m_1}\|_{L^2}\|\phi_{2,j_2,m_2,n_2}\|_{L^2}\\
&\lesssim\|\phi_1\|_{X_{0,\frac{1}{2}+}}\|\phi_2\|_{X_{0,\frac{1}{2}}}.
\end{split}
\end{equation*}
 If $n_2-m_2\geq2m_2\geq0$ and $j>2m_2$, since $\frac{|\mu_2|}{|\xi_2|}\ll\max(|\xi_1|^2,\frac{|\mu_1|}{|\xi_1|})$
and $|\xi_1|^2\gg\frac{|\mu_1|}{|\xi_1|}$, one can get
$(n_2-m_2)\leq 2m_1.$ Recalling that $s>0$, we obtain
\begin{equation*}
\begin{split}
(\ref{1})&\lesssim\sum_{j_1,j_2\geq 0}\sum_{0<2m_2\leq n_2-
m_2}\sum_{j\geq 2m_2\geq0}\sum_{m_1\geq m_2\geq 0}
2^{j(-\frac{1}{2}+)}2^{-m_1}2^{(n_2-m_2)/2}2^{m_2/2}2^{-(n_2-m_2)s}\\
&\times
2^{j_1/2}2^{j_2/2}\|g_j\|_{L^2}\|\phi_{1,j_1,m_1}\|_{L^2}\|\phi_{2,j_2,m_2,n_2}\|_{L^2}\\
&\lesssim\|\phi_1\|_{X_{0,\frac{1}{2}+}}\|\phi_2\|_{X_{0,\frac{1}{2}}}.
\end{split}
\end{equation*}
\noindent If $m_2<0$, we have
\begin{equation*}
\begin{split}
(\ref{1})&\sum_{j_1,m_1\geq 0}\sum_{ m_2< 0}\sum_{2m_1\geq
n_2-m_2\geq 0}
\sum_{2m_1\geq\max(j_2,j)\geq0}2^{j(-\frac{1}{2}+)}2^{-m_1}2^{(n_2-m_2)(1/2-s)}2^{j_1/2}2^{j_2/2}2^{m_2/2}\\
&\quad\quad\times
\|g_j\|_{L^2}\|\phi_{1,j_1,m_1}\|_{L^2}\|\phi_{2,j_2,m_2,n_2}\|_{L^2}\\
&\lesssim\sum_{j_1,m_1\geq
0}\sum_{2m_1\geq \max(j_2,j)\geq0}2^{j(-\frac{1}{2}+)}2^{-2m_1s}2^{j_1/2}2^{j_2/2}\\
&\quad\quad\times
\|g_j\|_{L^2}\|\phi_{1,j_1,m_1}\|_{L^2}\|\phi_{2,j_2}\|_{L^2}\\
&\lesssim\|\phi_1\|_{X_{0,\frac{1}{2}}}\|\phi_2\|_{X_{0,\frac{1}{2}}}.
\end{split}
\end{equation*}
We now consider the case $0\leq \max(j,j_2)\leq m_1$, $0\leq j<2m_2$
and $n_2-m_2>0$. It is clear that $j_1=\max(j,j_1,j_2)$ and
$|\xi_1|^4\lesssim 2^{j_1}$, since $2^{\max(j,j_1,j_2)}\gtrsim
|\xi_1|^4|\xi_2|$. We bound (\ref{1}) by
\begin{equation*}
\begin{split}
\sum_{j_1,j_2\geq0}\sum_{j\geq 0}&2^{j(-\frac{1}{2}+)}\int_{A_3}
g_j(\xi,\mu,\tau)\chi_j(\xi,\mu,\tau)\\
&\times|\xi_1+\xi_2|\hat\phi_{1,j_1}(\xi_1,\mu_1,\tau_1)
\frac{\hat\phi_{2,j_2}(\xi_2,\mu_2,\tau_2)}{\Big(1+|\xi_2|^2+\frac{|\mu_2|}{|\xi_2|}\Big)^s}.
\end{split}
\end{equation*}
There exists $\min(\frac{1}{2},s)>\delta>0$ small enough such that
\begin{equation*}
\begin{split}
(\ref{1})&\lesssim\sum_{m_1\geq 0}\sum_{j_1\geq 4m_1\geq
0}\sum_{j_1\geq 4m_1\geq 0}\sum_{2m_1\geq
\max(j_2,j)\geq0}2^{j(-\frac{1}{2}+)}2^{m_1(1/2+\delta)}\\
&\quad\quad\times
\||D_x|^{1/2-\delta}g_j^\vee\|_{L^\frac{2}{1-2\delta}_T(L^\frac{1}{\delta}_{x,y})}
\|\phi_{1,j_1,m_1}\|_{L^2}\||D_x|^{\delta}\phi_{2,j_2}\|_{L^{\frac{1}{\delta}}_T(L^\frac{2}{1-2\delta}_{x,y})}\\
&\lesssim\sum_{m_1\geq 0}\sum_{j_1\geq 4m_1\geq 0}\sum_{2m_1\geq
\max(j_2,j)\geq0}2^{j(-\frac{1}{2}+)}2^{m_1(1/2+\delta)}2^{j/2}2^{j_2/2}
\|g_j\|_{L^2}\|\phi_{1,j_1,m_1}\|_{L^2}\|\phi_{2,j_2}\|_{L^2}\\
&\lesssim\|\phi_1\|_{X_{0,\frac{1}{2}+}}\|\phi_2\|_{X_{0,\frac{1}{2}}}.
\end{split}
\end{equation*}

\vspace{0.3cm}

\noindent{\bf Case B:} $|\xi_1|^2\ll\frac{|\mu_1|}{|\xi_1|}.$

\vspace{0.3cm}

We first note that $|\mu_2|\ll|\mu_1|$, otherwise we have
$\frac{|\mu_2|}{|\xi_2|}\gtrsim\frac{|\mu_1|}{|\xi_1|}$, which is
contradiction with the assumption
$\frac{|\mu_2|}{|\xi_2|}\ll\frac{|\mu_1|}{|\xi_1|}$ and
$|\xi_2|\ll|\xi_1|.$ Thus we have $|\mu_1+\mu_2|\sim|\mu_1|$. The
argument in case A can be run smoothly until we come to the case
$|\mu_2|\geq1$ and $\max(j,j_2)< 2m_1$. We perform the change of
variables (\ref{eq:3.5}). It is easy to see that $|J_\mu|\gtrsim
|\xi_1|^4$. By the same estimate in (\ref{eq:3.11}), for fixed
$\theta_1,\theta_2,\xi_1,\xi_2,\mu_1$, the length of the symmetric
intervals where free variable $\mu_2$ can range is
$\Delta_{\mu_2}<2^{j-2m_1}$. Then we have
\begin{equation*}
\begin{split}
(\ref{2})&\lesssim\sum_{j_1,m_1\geq
0}\sum_{2m_1>\max(j_2,j)\geq0}2^{j(-\frac{1}{2}+)}2^{m_1}\int
g_j(u,v,w)\chi_j(u,v,w)\\
&\times|J_\mu|^{-1}H(u,v,w,\mu_2,\theta_1,\theta_2)dudvdwd\mu_2d\theta_1d\theta_2\\
&\lesssim \sum_{j_1,m_1\geq
0}\sum_{2m_1>\max(j_2,j)\geq0}2^{j(-\frac{1}{2}+)}2^{j/2-2m_1}2^{j_1/2}2^{j_2/2}\\
&\times \|g_j\|_{L^2}\|\phi_{1,j_1,m_1}\|_{L^2}
\|\phi_{2,j_2}\|_{L^2}\\
&\lesssim
\|\phi_1\|_{X_{0,\frac{1}{2}}}\|\phi_2\|_{X_{0,\frac{1}{2}}}.
\end{split}
\end{equation*}

\vspace{0.3cm}

\noindent {\bf Case C:} $|\xi_1|^2\sim \frac{|\mu_1|}{|\xi_1|}.$

\vspace{0.3cm}

Since
$\frac{|\mu_2|}{|\xi_2|}\ll|\xi_1|^2\sim\frac{|\mu_1|}{|\xi_1|},$ we
also have $|\mu_1+\mu_2|\sim|\mu_1|$. In this case, the resonant
interaction will happen. We bound (\ref{1}) and (\ref{2}) by
\begin{equation*}
\sum_{j_1,j_2\geq0}\sum_{j\geq 0}\int_{A_3}
g_j(\xi,\mu,\tau)\chi_j(\xi,\mu,\tau)|\xi_1+\xi_2|\hat\phi_{1,j_1}(\xi_1,\mu_1,\tau_1)
\frac{\hat\phi_{2,j_2}(\xi_2,\mu_2,\tau_2)}{\Big(1+|\xi_2|^2+\frac{|\mu_2|}{|\xi_2|}\Big)^s}.
\end{equation*}

\noindent We decompose $|\xi_1|\sim 2^{m_1}, m_1\geq 0$, and first
consider a special case $|\mu_2|<1$ and
$|\xi_2|\leq|\xi_1|^{-2-\varepsilon}$ for some $\varepsilon>0$ small
enough. In this case, we can use Proposition \ref{prop5}. (\ref{1})
and (\ref{2}) can be bounded by
\begin{equation*}
\begin{split}
&\sum_{j,j_1,j_2\geq 0}\sum_{m_1\geq
0}2^{j(-\frac{1}{2}+)}2^{m_1}\|g_j\|_{L^2}\|\phi_{1,j_1,m_1}\|_{L^\infty_T(L^2_{x,y})}
\|(m(\xi_2,\mu_2)\hat\phi_{2,j_2})^\vee\|_{L^2_T(L^\infty_{x,y})}\\
\lesssim&\sum_{j,j_1,j_2\geq 0}\sum_{m_1\geq
0}2^{j(-\frac{1}{2}+)}2^{-\frac{\varepsilon}{2} m_1}\|g_j\|_{L^2}2^{j_1/2}
\|\phi_{1,j_1,m_1}\|_{L^2}\|\phi_{2,j_2}\|_{L^2}\\
\lesssim&\|\phi_1\|_{X_{0,\frac{1}{2}}}\|\phi_2\|_{X_{0,\frac{1}{2}}}.
\end{split}
\end{equation*}
In the remaining estimates, we always have
$|\xi_2|>|\xi_1|^{-2-\varepsilon}$ for the same $\varepsilon$ as
above. In fact, $|\mu_2|>1$ implies
$|\xi_2|>|\xi_1|^{-2}>|\xi_1|^{-2-\varepsilon}$, since
$|\mu_2|\ll|\xi_1|^2|\xi_2|$.

Now we consider the case $\max(j,j_2)\geq (2-\varepsilon)m_1$ for
the same $\varepsilon$ as above. When $j=\max(j,j_2)$, there exists
$\min(\frac{1}{6},s)>\delta>0$ small enough such that
\begin{equation*}
\begin{split}
(\ref{1}),(\ref{2})&\lesssim\sum_{j_1,j\geq0}\sum_{0\leq\max( j_2,
(2-\varepsilon)m_1)\leq
j}2^{j(-\frac{1}{2}+)}2^{m_1(1/2+\delta)}2^{(2+\varepsilon)m_1\delta}
\|g_j\|_{L^2}\\
&\quad\quad\times\||D_x|^{\frac{1}{2}-\delta}\phi_{1,j_1,m_1}\|_
{L^\frac{2}{1-2\delta}_T(L^{\frac{1}{\delta}}_{x,y})}\||D_x|^{\delta}\phi_{2,j_2}
\|_{L^{\frac{1}{\delta}}_T(L^\frac{2}{1-2\delta}_{x,y})}\\
&\lesssim \sum_{j_1,j\geq0}\sum_{0\leq\max( j_2,
(2-\varepsilon)m_1)\leq
j}2^{j(-\frac{1}{2}+)}2^{m_1(1/2+(3+\varepsilon)\delta)}2^{j_1/2}2^{j_2/2}\\
&\quad\quad\quad\times\|g_j\|_{L^2}
\|\phi_{1,j_1,m_1}\|_{L^2}\|\phi_{2,j_2}\|_{L^2}\\
&\lesssim\|\phi_1\|_{X_{0,\frac{1}{2}}}\|\phi_2\|_{X_{0,\frac{1}{2}}}.
\end{split}
\end{equation*}
When $j_2=\max(j,j_2)$,
\begin{equation*}
\begin{split}
(\ref{1}),(\ref{2})&\lesssim\sum_{j_1,j_2\geq0}\sum_{j_2\geq
\max((2-\varepsilon)m_1,j)\geq0}2^{j(-\frac{1}{2}+)}2^{m_1/2}
\||D_x|^\frac{1}{4}g_j^\vee\|_{L^4_T(L^4_{x,y})}\||D_x|^\frac{1}{4}\phi_{1,j_1,m_1}\|_{L^4_T(L^4_{x,y})}\\
&\quad\quad\times\|\phi_{2,j_2}\|_{L^2}\\
&\lesssim\sum_{j_1,j_2\geq0}\sum_{j_2\geq
\max(j,(2-\varepsilon)m_1)\geq0}2^{j(-\frac{1}{2}+)}2^{j/2}2^{m_1/2}2^{j_1/2}
\|g_j\|_{L^2}\|\phi_{1,j_1,m_1}\|_{L^2}\|\phi_{2,j_2}\|_{L^2}.\\
&\lesssim
\|\phi_1\|_{X_{0,\frac{1}{2}}}\|\phi_2\|_{X_{0,\frac{1}{2}+}}.
\end{split}
\end{equation*}

\noindent At last, we consider the case
$\max(j,j_2)<(2-\varepsilon)m_1$ for the same $\varepsilon$ as
above.

\vspace{0.3cm}

\noindent {\bf Subcase 1:}
$\Big|\frac{\mu_1}{\xi_1}-\frac{\mu_2}{\xi_2}\Big|^2<\frac{1}{2}|\xi_1+\xi_2|^2
|5(\xi_1^2+\xi_1\xi_2+\xi_2^2)-3\alpha|.$

\vspace{0.3cm}

\noindent Since now the resonant interaction does not happen, we
have $|\xi_1|^4|\xi_2|\leq 2^{\max(j,j_1,j_2)}$. And because
$|\xi_2|>|\xi_1|^{-2-\varepsilon}$, we get that
$j_1=\max(j,j_1,j_2)\geq (2-\varepsilon)m_1$. By choosing
$\min(\frac{1}{6},s)>\delta>0$ small enough, we have
\begin{equation*}
\begin{split}
(\ref{1}),(\ref{2})&\lesssim\sum_{j_1\geq0}\sum_{0\leq \max(j,
j_2)\leq (2-\varepsilon)m_1\leq
j_1}2^{j(-\frac{1}{2}+)}2^{m_1(1/2+\delta)}2^{(2+\varepsilon)m_1\delta}
\|\phi_{1,j_1,m_1}\|_{L^2}\\
&\quad\quad\times\||D_x|^{\frac{1}{2}-\delta}g_j^\vee\|_{L^\frac{2}{1-2\delta}_T(L^{\frac{1}{\delta}}_{x,y})}\||D_x|^{\delta}\phi_{2,j_2}\|_{L^{\frac{1}{\delta}}_T(L^\frac{2}{1-2\delta}_{x,y})}\\
&\lesssim \sum_{j_1\geq0}\sum_{0\leq \max(j, j_2)\leq
(2-\varepsilon)m_1\leq
j_1}2^{j(-\frac{1}{2}+)}2^{m_1(1/2+(3+\varepsilon)\delta)}2^{j/2}2^{j_2/2}\|g_j\|_{L^2}
\|\phi_{1,j_1,m_1}\|_{L^2}\|\phi_{2,j_2}\|_{L^2}\\
&\lesssim\|\phi_1\|_{X_{0,\frac{1}{2}+}}\|\phi_2\|_{X_{0,\frac{1}{2}}}.
\end{split}
\end{equation*}

\vspace{0.3cm}

\noindent {\bf Subcase 2:}
$\Big|\frac{\mu_1}{\xi_1}-\frac{\mu_2}{\xi_2}\Big|^2\geq\frac{1}{2}|\xi_1+\xi_2|^2
|5(\xi_1^2+\xi_1\xi_2+\xi_2^2)-3\alpha|$.

\vspace{0.3cm}

\noindent As we know, we also need to consider two cases.
\begin{equation}\label{eq:3.15}\Big|5(\xi_1^4-\xi_2^4)-3\alpha(\xi_1^2-\xi_2^2)-
\Big[\Big(\frac{\mu_1}{\xi_1}\Big)^2-\Big(\frac{\mu_2}{\xi_2}\Big)^2\Big]\Big|\geq
|\xi_1|^{\frac{1}{2}},\end{equation} and
\begin{equation}\label{eq:3.16}\Big|5(\xi_1^4-\xi_2^4)-3\alpha(\xi_1^2-\xi_2^2)-
\Big[\Big(\frac{\mu_1}{\xi_1}\Big)^2-\Big(\frac{\mu_2}{\xi_2}\Big)^2\Big]\Big|<|\xi_1|^{\frac{1}{2}}.\end{equation}
(\ref{eq:3.15}) means the determinant of Jacobian of the change of
variables (\ref{eq:3.5}), $|J_\mu|\geq |\xi_1|^\frac{1}{2}$. Thus we
have
\begin{equation*}
\begin{split}
(\ref{1}),(\ref{2})&\lesssim
\sum_{m_1,j_1\geq0}\sum_{(2-\varepsilon)m_1>\max(j_2,j)\geq0}2^{j(-\frac{1}{2}+)}
2^{j/2-m_1}2^{m_1(1-\frac{1}{4})}2^{j_1/2}2^{j_2/2}\\
&\quad\quad\quad\times\|g_j\|_{L^2}
\|\phi_{1,j_1,m_1}\|_{L^2}\|\phi_{2,j_2}\|_{L^2}.\\
&\lesssim\|\phi_1\|_{X_{0,\frac{1}{2}}}\|\phi_2\|_{X_{0,\frac{1}{2}}}.
\end{split}
\end{equation*}
When (\ref{eq:3.16}) occurs, we recur to the change of variables
(\ref{eq:3.7}). By the argument in (\ref{eq:3.17}) and
(\ref{eq:3.18}), for fixed $\theta_1,\theta_2,\xi_2,\mu_1,\mu_2$,
the length of the interval where $\xi_1$ ranges is
$|\xi_1|<2^{(\frac{1}{2}-3)m_1}$. Thus we obtain
\begin{equation*}
\begin{split}
(\ref{1}),(\ref{2})&\lesssim
\sum_{m_1,j_1\geq0}\sum_{(2-\varepsilon)m_1>\max(j_2,j)\geq0}2^{j(-\frac{1}{2}+)}2^{-(2-\frac{1}{2})m_1}
2^{m_1}2^{j_1/2}2^{j_2/2}\\
&\quad\quad\quad\times\|g_j\|_{L^2}
\|\phi_{1,j_1,m_1}\|_{L^2}\|\phi_{2,j_2}\|_{L^2}.\\
&\lesssim\|\phi_1\|_{X_{0,\frac{1}{2}}}\|\phi_2\|_{X_{0,\frac{1}{2}}}.
\end{split}
\end{equation*}
We now finish the proof of Theorem \ref{biestimate}.

\end{proof}

\hspace{5mm}

\hspace{5mm}

\section{Proof of Main Theorem}

We now state the proof of Theorem \ref{th:2}.

\vspace{0.3cm}

\begin{proof}: Considering the integral equation according to
(\ref{5KP})
\begin{eqnarray}\label{eq:4.1}
u(t)=\psi(t)\Big[S(t)u_0-\frac{1}{2}
\int_0^tS(t-t')\partial_x(\psi_T^2(t')u^2(t'))dt'\Big],
\end{eqnarray}
where $0<T<1$, and $\psi_T(t)$ is the same bump function with
(\ref{bump}). It is clear that a solution for (\ref{eq:4.1}) is a
fixed point of the nonlinear operator
\begin{equation}\label{eq:4.2}
L(u)=\psi(t)S(t)u_0-\frac{1}{2}\psi(t)\int_0^t
S(t-t')\partial_x(\psi_T^2(t')u^2(t'))dt'.
\end{equation}
Thus we need to prove the operator $L$ is a contractive mapping from
the following closed set to itself
\begin{eqnarray}
B_a=\Big\{u\in X_{s,\,b},\,\|u\|_{X_{s,b}}\leq a\Big\},
\end{eqnarray}
where $a=4C\|u_0\|_{E_s}$. By Proposition \ref{prop1} and Theorem
\ref{biestimate}, there exist $\sigma
>0$ such that
\begin{eqnarray}\label{eq:4.5}
\|L(u)\|_{X_{s,\frac{1}{2}+}}\leq
C\|u_0\|_{E_s}+CT^\sigma\|u\|_{X_{s,\frac{1}{2}+}}^2.
\end{eqnarray}
Next, since
$\partial_x(u^2)-\partial_x(v^2)=\partial_x[(u-v)(u+v)]$, we get in
the same way that
\begin{eqnarray}\label{eq:4.3}
\|L(u)-L(v)\|_{X_{s,\,\frac{1}{2}+}}&\leq &
CT^\sigma\|u-v\|_{X_{s,\,\frac{1}{2}+}}
\Big(\|u\|_{X_{s,\,\frac{1}{2}+}}+\|v\|_{X_{s,\,\frac{1}{2}+}}\Big).
\end{eqnarray}
By choosing $T=T(\|u_0\|_{E_s})$ such that
$8CT^\sigma\|u_0\|_{E_s}<1$, we deduce that from (\ref{eq:4.5}) and
(\ref{eq:4.3}) that $L$ is strictly contractive on the ball $B_a$.
Thus, there exists unique solution to the IVP of the fifth order
KP-I equation $u\in X_{s,\frac{1}{2}+}$ on the interval $[-T,T]$.
The smoothness of the mapping from $E_s$ to $X_{s,\frac{1}{2}+}$
follows from the fixed point argument. Since
$X_{s,\frac{1}{2}+}\subset C([-T,T];E_s)$, we finish the proof of
Theorem \ref{th:2}.
\end{proof}

\vspace{0.5cm}

\noindent {\bf Acknowledgements}

\vspace{0.3cm}

The authors would like to express their gratitude to the anonymous
referee and the associated editor for their invaluable comments
and suggestions which helped improve the paper greatly. J.Li is
also very grateful to M.\,Ben-Artzi for sharing us their research
paper \cite{BenSau}. All the authors are supported by the NNSF of
China (No.10725102, No.10771130, No.10626008).

\hspace{5mm}
\begin{center}
\end{center}

\end{document}